\documentclass[a4paper,12pt,reqno]{amsart}
\usepackage{amsmath,amsfonts,amssymb,amsthm,enumerate,multicol}
\usepackage{mathrsfs}
\usepackage{tikz,graphicx}
\usepackage{hyperref}
\usepackage{caption,enumitem,wasysym}
\usepackage{cleveref}
\usepackage{epstopdf}
\usepackage{float,wrapfig}
\usepackage[labelsep=space]{caption}
\usepackage[font=footnotesize]{caption}
\usepackage{xcolor}
\usepackage{cases}
\usepackage[english]{babel}
\usepackage[autostyle]{csquotes}
\usepackage{cite}

\theoremstyle{plain}
\newtheorem{thm}{Theorem}[section]

\newtheorem{lem}[thm]{Lemma}
\newtheorem{prop}[thm]{Proposition}

\newtheorem{defn}[thm]{Definition}
\newtheorem{rem}[thm]{Remark}

\usepackage{mathtools}

\usepackage{amsmath,amsfonts,amssymb,amsthm,enumerate,multicol}
\usepackage{tikz}
\usepackage{float}
\usepackage{autobreak}

\allowdisplaybreaks

\numberwithin{equation}{section}

\def\ff{u}
\def\kk{\Phi}
\def\bb{\beta}
\def\sss{\varphi}

\def\mm{\boldsymbol\mu}
\def\tf{\theta}

\begin{document} 
	\begin{center}
		\Large{\textbf{Mass-conserving weak solutions to the continuous  nonlinear fragmentation equation in the presence of mass transfer}}
	\end{center}

	\medskip
\medskip
\centerline{${\text{ ${\text{Ram Gopal~ Jaiswal$^{\dagger}$}}$ and  Ankik Kumar Giri$^{\dagger*}$}}$}\let\thefootnote\relax\footnotetext{$^{*}$Corresponding author. Tel +91-1332-284818 (O);  Fax: +91-1332-273560  \newline{\it{${}$ \hspace{.3cm} Email address: }}ankik.giri@ma.iitr.ac.in}
\medskip
{\footnotesize

	\centerline{ ${}^{}$  $\dagger$ Department of Mathematics, Indian Institute of Technology Roorkee,}
	\centerline{Roorkee-247667, Uttarakhand, India}
	
}
	\bigskip
	
	\begin{quote}
	{\small {\em \bf Abstract.} 
	A mathematical model for the continuous nonlinear fragmentation equation is considered in the presence of mass transfer. In this paper, we demonstrate the existence of mass-conserving weak solutions to the nonlinear fragmentation equation with mass transfer for collision kernels of the form $\kk(x,y) = \kappa(x^{{\sigma_1}} y^{{\sigma_2}} + y^{{\sigma_1}} x^{{\sigma_2}})$, $\kappa>0$, $0 \leq {\sigma_1} \leq {\sigma_2} \leq 1$, and ${\sigma_1} \neq 1$ for $(x, y) \in \mathbb{R}_+^2$, with integrable daughter distribution functions, thereby extending previous results obtained by Giri \& Lauren\c cot (2021). In particular, the existence of at least one global weak solution is shown when the collision kernel exhibits at least linear growth, and one local weak solution when the collision kernel exhibits sublinear growth. In both cases, finite superlinear moment bounds are obtained for positive times without requiring the finiteness of initial superlinear moments. Additionally, the uniqueness of solutions is confirmed in both cases.
}
	\end{quote}
	
	\vspace{0.5cm}
	
	

		\section{\textbf{Introduction}}\label{sec:intro}
The continuous nonlinear fragmentation equation with mass transfer was first introduced in astrophysics \cite{safronov1972} and atmospheric science \cite{list1976} to study the evolution of droplet size distributions in clouds. The nonlinear fragmentation equations,  commonly referred to as the collision-induced breakage equation, characterize the behavior of a particle system in which particles break apart due to binary collisions, resulting in the creation of daughter particles with possible transfer of mass. Let $\ff(t,x)$ denote the density of particles of size $x \in \mathbb{R}_+:=(0,\infty)$ at time $t \geq 0$. The time evolution of $\ff$ is described by the following nonlinear integro-partial differential equation
	\begin{align}
		\partial_{t} \ff(t,x) =&\frac{1}{2} \int_{x}^{\infty} \int_{0}^{y} \bb(x,y-z,z) \kk(y-z,z) \ff(t,y-z) \ff(t,z) \, dz \, dy\nonumber \\
	&	- \int_{0}^{\infty} \kk(x,y) \ff(t,x) \ff(t,y) \, dy, \quad (t,x) \in \mathbb{R}_+^2, \label{main:eq}
	\end{align}
	with initial condition
	\begin{equation}
		\ff(0,x) = \ff^{\mathrm{in}}(x), \quad x \in \mathbb{R}_+, \label{in:eq}
	\end{equation}
	where the collision kernel $\kk(x,y) = \kk(y,x) \geq 0$ represents the rate of collision between two particles of size $x$ and $y$, while the daughter distribution function (also referred to as the breakage kernel) $\bb(z,x,y) = \bb(z,y,x) \geq 0$ represents the rate at which particles of size $z$ are generated from a collision event between two particles of size $x$ and $y$. The first term on the right-hand side of equation \eqref{main:eq} accounts for the generation of particles of size $x$ due to collisions between two particles of size $y-z$ and $z$, while the second term represents the reduction in the number of particles of size $x$ as a result of their collisions with particles of any size.
	
	Assuming that mass is conserved locally, we impose the following condition
	\begin{equation}
		\int_{0}^{x+y} z \bb(z,x,y) \, dz = x + y \quad \text{and} \quad \bb(z,x,y) = 0 \, ~\text{for} ~\, z > x + y, \label{localc:eq}
	\end{equation}
	for all $(x,y) \in \mathbb{R}_+^2$. This condition leads to the global mass conservation principle, stated as
	\begin{equation}
		\int_{0}^{\infty} x \ff(t,x) \, dx = \int_{0}^{\infty} x \ff^{\text{in}}(x) \, dx, \quad t \geq 0. \label{globalc:eq}
	\end{equation}
Additionally, we assume that the breakage kernel $\bb$ is integrable, as this guarantees that the total number of daughter particles formed due to the collision of particles of sizes $x$ and $y$ remains finite. Mathematically, this condition is expressed {as in \cite{ali2024, GL12021, GL22021}} for a constant $\gamma\ge 2$
\begin{equation}
2 \le	\mathrm{N}(x,y) := \int_0^{x+y} \bb(z,x,y) \, dz \le \gamma, \quad (x,y) \in \mathbb{R}_+^2.\label{nop_assump:eq}
\end{equation}
Furthermore, from equation \eqref{localc:eq}, we deduce that the collision between particles of sizes $x$ and $y$ cannot create particles larger than their combined size $x + y$. However, it is possible that a collision between particles of sizes $x$ and $y$ might generate a fragment larger than $\max\{x,y\}$ due to mass transfer from the smaller to the larger particle. For instance, the collision between particles of size $x$ and $y$ could produce one fragment of size $(2y + x)/2$ and another fragment of size $x/2$. However, if there exists a function $\bar{\bb}\ge 0$ such that
\begin{equation}
	\bb(z,x,y) = \bar{\bb}(z,x,y) \mathbf{1}_{(0,x)}(z) + \bar{\bb}(z,y,x) \mathbf{1}_{(0,y)}(z), \label{masstransfer:eq}
\end{equation}
and
\begin{equation}
	\int_{0}^{x} z \bar{\bb}(z,x,y) \, dz = x, \quad \text{and} \quad \bar{\bb}(z,x,y)=0 ~~ \text{for} ~~z>x, \label{Blocalc:eq}
\end{equation}

for all $(x,y,z) \in \mathbb{R}_+^3$, then no mass transfer will occur during the collision$ $. Specifically, consider two particles with respective sizes $x$ and $y$ that collide with collision rate  $\kk(x,y)$. Upon collision, the particle of size $x$ produces fragments of size $z<x$ according to the distribution function $\bar{\bb}(z,x,y)$, while the particle of size $y$ produces fragments of size $z<y$ according to the distribution function $\bar{\bb}(z,y,x)$. In this scenario, the transfer of mass from a particle of size $x$ to a particle of size	$y$ would be absent, as the particle sizes after the collision cannot exceed their original sizes $x$ and $y$.  This leads to the nonlinear fragmentation equation without mass transfer \cite{cheng1988,cheng1990,ernst2007, krapivsky2003,kostoglou2000,kostoglou2006}, given by
	\begin{align}
		\partial_{t} \ff (t,x) = &\int_{0}^{\infty}\int_{x}^{\infty}\bar{\bb} (x,y,z)\kk(y,z)\ff (t,y)\ff (t,z) \,d{y} \,d{z} \nonumber \\
		&- \int_{0}^{\infty}\kk(x,y)\ff (t,x)\ff(t,y) \,d{y}, \qquad (t,x) \in \mathbb{R}_+^2.
		\label{main without mass transfer}
	\end{align}
	The nonlinear fragmentation equation without mass transfer \eqref{main without mass transfer} can be derived from \eqref{main:eq} by applying the conditions \eqref{masstransfer:eq} and \eqref{Blocalc:eq}.

\medskip
The linear fragmentation equation, initially studied by Filippov \cite{Fil1961}, Kapur \cite{kapur1972}, and McGrady and Ziff \cite{ziff1987, ziff1991}, has garnered significant attention over the last few decades. This equation has been further analyzed using deterministic approaches in \cite{banasiak2002, banasiak2004, banasiak2006, bll2019, bt2018, eme2005}, and through probabilistic methods in \cite{bert2002, haas2003}. Comprehensive accounts of this topic can be found in the books \cite{bll2019, bert2006}.

In comparison to linear fragmentation, the nonlinear fragmentation equation \eqref{main:eq} has received relatively less attention. However, recent studies have established several mathematical results for the nonlinear fragmentation equation without mass transfer \eqref{main without mass transfer}, as presented in \cite{GL12021, GL22021, GJL2024, JA2024}. The nonlinear fragmentation equation without mass transfer \eqref{main without mass transfer} has been widely explored in the physics literature, with particular emphasis on scaling properties and the shattering transition, where mass is lost due to the formation of ``particles of zero size" or dust \cite{cheng1988, cheng1990, ernst2007, kostoglou2000, kostoglou2006, krapivsky2003}. Cheng and Redner \cite{cheng1988, cheng1990} conducted an asymptotic analysis of nonlinear fragmentation events in models where a collision between two particles leads to either both particles splitting into equal halves or one of the particles (larger or smaller) dividing into two. They focused on the self-similar solutions for the evolution $\ff$ at sufficiently large times, expressed as
\begin{align}
	\ff(t,x) & \sim \frac{1}{e(t)^2} \sss\left(\frac{x}{e(t)}\right), \label{self-similar form}
\end{align}
where the dynamics are governed by the size parameter $e(t)$ of particles at time $t$ and a self-similar profile $\sss$. Later, the evolution of nonlinear fragmentation was analyzed in \cite{krapivsky2003} by exploring its traveling wave behavior to obtain closed-form solutions for the mass distribution. Kostoglou and Karabelas \cite{kostoglou2000, kostoglou2006} transformed \eqref{main without mass transfer} into a linear fragmentation equation to derive analytical solutions for specific collision kernels. For the constant collision kernel $\kk(x,y) = 1$, they obtained solutions defined for finite time, while for the classical multiplicative kernel $\kk(x,y) = xy$, the solutions were defined for all time. Furthermore, they analyzed the self-similar profile of \eqref{main without mass transfer} for the multiplicative collision kernel $\kk(x,y) = (xy)^{\sigma/2}$ and the sum collision kernel $\kk(x,y) = x^\sigma + y^\sigma$ for $\sigma > 1$, and for the distribution function $ \bar{\bb} ({z},{x},{y})=(\nu+2){{z}}^\nu {{x}}^{-\nu-1}\textbf{1}_{(0,{x})}({z}),$ with $\nu\in (-2,0]$. In addition, the asymptotic behavior of solutions for the nonlinear fragmentation equation \eqref{main without mass transfer} studied in \cite{ernst2007} by transferring it to the linear fragmentation equation with a different time scale for a product kernel $\kk({x},{y}) = ({x}{y})^{\sigma/2}$, where $0\le \sigma \le 2$ and study the occurrence of shattering phenomenon. In \cite{GL22021}, the well-posedness of the continuous nonlinear fragmentation equation without mass transfer \eqref{main without mass transfer} is discussed for integrable breakage kernels and collision kernels of the form {$\kk(x,y) = \kappa(x^{{\sigma_1}} y^{{\sigma_2}} + x^{{\sigma_2}} y^{{\sigma_1}})$, where $\kappa>0$ and $\sigma := {\sigma_1} + {\sigma_2} \in [0,2]$.} It is also shown that no mass-conserving weak solutions exist for $\sigma_1<0$, even for short times. Later, the continuous nonlinear fragmentation equation without mass transfer \eqref{main without mass transfer} with non-integrable breakage kernels has been considered in \cite{GJL2024} which extends all the results established in \cite{GL12021} to the non-integrable breakage kernels. Specifically, for a fixed $m_0\in(0,1)$, the global existence of mass-conserving weak solutions is proven for a class of collision kernels given by
	\begin{align*}
		\kk({x},{y})={x}^{\sigma_1} {y}^{\sigma_2}+{x}^{\sigma_2} {y}^{\sigma_1}, \quad  m_0\le {\sigma_1} \le {\sigma_2} \le 1, \quad  ({x},{y})\in \mathbb{R}_+^2,   \label{kernel nonintegrable}
	\end{align*}
	for $1 \leq \sigma := \sigma_1 + \sigma_2 \leq 2$. In addition, for $2m_0 \leq \sigma < 1$, it is proven that a mass-conserving weak solution exists, but this solution only persists for a finite time interval. Furthermore, the non-existence of mass-conserving weak solutions, even over a small time interval, is proven for $\sigma_1<m_0$. The existence of a mass-conserving self-similar profile of the form \eqref{self-similar form} for the nonlinear fragmentation equation without mass transfer has been discussed recently in \cite{JA2024} for a specific class of homogeneous collision kernels $\kk$ of the form $\kk({x},{y})={x}^{\sigma_1} {y}^{\sigma_2}+{x}^{\sigma_2} {y}^{\sigma_1},$ where $m_0 \le {\sigma_1} \le {\sigma_2} \le 1$ for a fixed $m_0\in[0,1)$ and $\sigma \in (1,2],$ and homogeneous breakage kernels $\bb$ that satisfy
\begin{align}
	\bb({z},{x},{y})=\frac{1}{{x}}{\widetilde{\bb}}\bigg(\frac{{z}}{{x}}\bigg) \textbf{1}_{(0,{x})}({z})+\frac{1}{{y}}{\widetilde{\bb}}\bigg(\frac{{z}}{{y}}\bigg) \textbf{1}_{(0,{y})}({z}), \qquad {({x},{y},z)\in\mathbb{R}_{+}^2},
\end{align}

subject to the condition

\begin{align}
	\int_0^1 z_*{\widetilde{\bb}}(z_*)dz_*=1. \label{eta_local_conservation}
\end{align}

For the nonlinear fragmentation equation with mass transfer \eqref{main:eq}, which was initially introduced in \cite{safronov1972, list1976}, analytical progress has been limited. In particular, only one analytical solution has been identified in the literature for a constant collision kernel and a specific form of the breakage kernel \cite{FTL1988}. The first rigorous mathematical investigation of the {continuous} nonlinear fragmentation equation with mass transfer was carried out in \cite[Theorem 2.5]{GL12021}, where the existence and uniqueness of global mass-conserving solutions were established for collision kernels satisfying
\begin{align}
	0\le \kk(x,y)\le\kappa(x+y), \qquad \kappa>0 \quad \text{and} \quad (x,y)\in \mathbb{R}_+^2,
\end{align}
with integrable breakage kernels.  Furthermore, a high-order discontinuous Galerkin technique has been used to efficiently solve the nonlinear fragmentation equation with mass transfer in \cite{lombart2024}. {A rigorous mathematical study on the discrete version of the nonlinear fragmentation equation in the presence of mass transfer is recently done in \cite{ali2024} where the existence and uniqueness of weak solutions are established. Moreover,  in \cite{ali2024}, the existence of stationary solutions is shown by a dynamical approach. In the present work, we mainly focus on showing the existence and uniqueness of weak solutions to the continuous nonlinear fragmentation equations with mass transfer which is motivated from \cite{ali2024, GL12021, GL22021}. This extends the results obtained in \cite[Theorem 2.5]{GL12021} by focusing on a broader class of collision kernels and integrable daughter distribution functions.} In particular, we demonstrate that global weak solutions exist for collision kernels with at least linear growth, while only local weak solutions are found when the collision kernels exhibit sublinear growth. Following the methodology in \cite{GL22021}, we apply the weak $L^1$-compactness technique introduced by Stewart in \cite{stewart1989} for coagulation-fragmentation
equations; for detailed discussions and additional references on this technique,
see~\cite{bll2019}. The proof of uniqueness follows the same approach as the uniqueness results provided in \cite[Proposition 1.6]{GL22021} for \eqref{main:eq} without mass transfer.

\medskip
Specifically, we focus our analysis on the collision kernel define by
	\begin{align}
	\kk(x, y) = \kappa (x^{{\sigma_1}} y^{{\sigma_2}} + y^{{\sigma_1}} x^{{\sigma_2}}), \qquad \text{for } (x, y) \in \mathbb{R}_+^2, \label{kernel:eq}
\end{align}
where $ 0 \le {\sigma_1} \le {\sigma_2} \le 1 $, $ {\sigma_1} \neq 1 $, $ \kappa > 0 $, and $ \sigma := {\sigma_1} + {\sigma_2} \in [0,2) $. For the breakage kernel,  $\bb$, we assume that there exists a function $ \eta$ from $\mathbb{R}_+$ to $\mathbb{R}_+$ which is non-decreasing and a parameter $ \alpha \in (0, 1) $ such that
\begin{align}
	\lim_{\delta \to 0} \eta(\delta) = 0, \label{ui_1_assump:eq}
\end{align}
and
\begin{align}
	\int_0^{x + y} \mathbf{1}_{E}(z) \bb(z, x, y) \, dz \le \eta(|E|) (x^{-\alpha} + y^{-\alpha}), \quad (x, y) \in \mathbb{R}_+^2, \label{ui_2_assump:eq}
\end{align}
for any measurable set $ E \subset \mathbb{R}_+ $ with $|E|<\infty$, where $|E|$ denotes the Lebesgue measure of $E$. In addition, we assume that, for each $ m > 1 $, there exist constants $ \varkappa_m \in (0,1) $ and $ \varsigma_m \ge 1 $ such that
\begin{align}
	\int_0^{x + y} z^m \bb(z, x, y) \, dz \le (1 - \varkappa_m) (x^m + y^m) + \varsigma_m (x y^{m-1} + y x^{m-1}), \label{mth nop_assump:eq}
\end{align}
for all $ (x, y) \in \mathbb{R}_+^2 $. {The assumption \eqref{mth nop_assump:eq} on $\bb$ is analogous to the condition imposed on the daughter distribution function in the discrete nonlinear fragmentation equation studied in \cite{ali2024} to obtain higher moment estimates.}
\begin{rem}
	Hypothesis \eqref{kernel:eq} is introduced to simplify the presentation. However, it is likely that the analysis provided below also true for a collision kernel $\kk = \kk_1 + \cdots + \kk_n$, where each of the functions $\kk_n$ for $i \in \{1,\cdots,n\}$, satisfy the assumption \eqref{kernel:eq} with same homogeneity $\sigma$. In addition, the analysis performed below equally applies to collision
	kernels being bounded from above and below by multiples of the collision kernel defined by \eqref{kernel:eq}. In particular, our results are also true for the collision kernel occur in an infinite system of gravitationally attracting, randomly distributed particles with a Maxwellian velocity distribution \cite{van1987,ziff1980}
	\begin{align*}
		\kk(x,y)= (xy)^{1/2}(x+y)^{1/2}(x^{1/3}+y^{1/3}), \qquad (x,y)\in \mathbb{R}_+^2,
	\end{align*}
	which can be written as
	\begin{align*}
	\frac{1}{\sqrt{2}}(\kk_1(x,y)+\kk_2(x,y))\le \kk(x,y) \le \kk_1(x,y)+\kk_2(x,y), \qquad (x,y)\in \mathbb{R}_+^2,
	\end{align*}
	where $\kk_1(x,y)= {x}^{1/2} {y}^{4/3}+{x}^{4/3} {y}^{1/2}$ and $\kk_2(x,y)= {x}^{5/6} {y}+{x} {y}^{5/6}$ satisfy the assumption \eqref{kernel:eq}.
\end{rem}
	\begin{rem}\label{rem:ex}
	Consider a breakage kernel defined by power law
	\begin{align}
		\bb_\nu(z, x, y) = (\nu + 2) \frac{z^\nu}{(x + y)^{\nu + 1}} \mathbf{1}_{(0, x + y)}(z), \qquad (x, y, z) \in \mathbb{R}_+^3, \label{example:eq}
	\end{align}
	where $\nu\in (-2,0]$, see \cite{FTL1988, vigil2006, GL12021}.
	The function $ \bb_\nu $ satisfies assumptions \eqref{localc:eq}, \eqref{nop_assump:eq}, \eqref{ui_1_assump:eq}, \eqref{ui_2_assump:eq} and \eqref{mth nop_assump:eq} for $ \nu \in (-1, 0] $ with
	\begin{align*}
		\eta(\delta) &= (\nu + 2) \left( \frac{p - 1}{p (\nu + 1) - 1} \right)^{(p - 1)/p} \delta^{1/p}, \quad \alpha = \frac{1}{p}, \quad 	\gamma = \frac{\nu + 2}{\nu + 1},
	\end{align*}
	provided that the parameter $ p > \frac{1}{\nu + 1} $ is sufficiently large. {In addition, for $m>1$ and $ \nu \in (-1, 0] $, it follows from \cite[Lemma 7.4.2.]{bll2019} that for $ C_m := 2^{m-1} - 1 $ for $ m \in (1,2] \cup [3, \infty) $ and $ C_m := m $ for $ m \in (2,3) $,
   \begin{align*}
   \int_0^{x+y}z^m\bb_\nu(z,x,y)dz&=\int_0^{x+y}z^m(\nu + 2) \frac{z^\nu}{(x + y)^{\nu + 1}}dz=\frac{\nu + 2}{m + \nu + 1}(x+y)^m\\
   &\le \frac{\nu+2}{m+\nu+1}(x^m+y^m)+\frac{C_m (\nu + 2)}{m + \nu + 1}(xy^{m-1}+x^{m-1}y),
   \end{align*}
   which implies that $\bb_\nu$ satisfies \eqref{mth nop_assump:eq} with $\varkappa_m = \frac{m - 1}{m + \nu + 1}$ and $\varsigma_m = \frac{C_m (\nu + 2)}{m + \nu + 1}.$}
	
\end{rem}
We now introduce the functional framework utilized in this paper. For a measurable function $W\ge0$ defined on $\mathbb{R}_+$, we denote $\Xi_W$ as $L^1(\mathbb{R}_+, W(x) \, dx)$, and the associated norm and moment are given by
\begin{align*}
	\| g \|_{W} & := \int_{0}^{\infty} |g(x)| \, W(x) \, dx, \\
	\mm_{W}(g) & := \int_{0}^{\infty} g(x) \, W(x) \, dx, \quad g \in \Xi_{W}.
\end{align*}
The space $\Xi_W$ equipped with the weak topology is denoted by $\Xi_{W, w}$, and its positive cone is represented as  $\Xi_{W, +}$. Specifically, when $W(x) = W_m(x) := x^m$ for some $m \in \mathbb{R} $, we denote $ \Xi_m := \Xi_{W_m} $ and define the norms and moments as
\begin{align*}
	\| g \|_{m} & := \int_{0}^{\infty} x^m |g(x)| \, dx, \\
	\mm_{m}(g) & := \mm_{W_m}(g) := \int_{0}^{\infty} x^m g(x) \, dx, \quad g \in \Xi_{m}.
\end{align*}

Now, we outline the content of the paper. \Cref{sec:main} includes the definition of a weak solution {to \eqref{main:eq}-\eqref{in:eq} and }presents the main results of the paper. \Cref{sec:existence} contains the proof of the existence of weak solutions for two cases: (1) for a class of restricted integrable breakage kernels with $\ff^{\mathrm{in}}\in \Xi_0 \cap \Xi_{1,+}$ (\Cref{thm:existence E=0}), and (2) for integrable breakage kernels with a restricted set of initial data, i.e., $\ff^{\mathrm{in}}\in \Xi_{-\alpha} \cap \Xi_{1,+}$ for $\alpha \in (0,1)$ (\Cref{thm:existence}), using a weak $L^{1}$ compactness approach. Initially, we truncate equations \eqref{main:eq}-\eqref{in:eq} using a truncated collision kernel and demonstrate that this formulation is well-posed by applying the Banach fixed-point theorem. Consequently, we obtain a sequence of solutions for which we establish sublinear moments, emphasizing the significant role of the lower bound on the ${\sigma_1}$-moment in estimating the superlinear moments, which are shown to be finite for all positive times. Subsequently, we apply the Dunford-Pettis theorem to confirm that the sequence of solutions is weakly compact with respect to size, supported by a uniform integrability estimate. In the next step, we evaluate time equicontinuity to achieve compactness concerning time. 



	\section{Main results}\label{sec:main}
	To present the main results of the paper, we first define the concept of a weak solution for \eqref{main:eq}-\eqref{in:eq}.
	
	\begin{defn}[{Weak solutions}]\label{defn:weaksolution}
		Consider the breakage kernel $\bb$ satisfying \eqref{localc:eq} and \eqref{nop_assump:eq}. Let the initial data $ \ff^{\mathrm{in}} \in \Xi_{0,+} \cap \Xi_{1} $ and $ T \in (0, \infty] $. A non-negative function $ \ff $ is said to be a weak solution to \eqref{main:eq}-\eqref{in:eq} on $ [0, T) $ if
		\begin{equation}
			\ff \in \mathcal{C}([0, T), \Xi_{0, w}) \cap L^{\infty}((0, T), \Xi_{1}), \label{wf1:eq}
		\end{equation}
		with $ \ff(0) = \ff^{\mathrm{in}} $ on $ \mathbb{R}_+ $,
		\begin{equation}
			(\tau, x, y) \mapsto \kk(x, y) \ff(\tau, x) \ff(\tau, y) \in L^{1}((0, t) \times \mathbb{R}_+^2), \label{wf2:eq}
		\end{equation}
		and
		\begin{align}
			\int_{0}^{\infty} \tf(x) \ff(t, x) \, dx &= \int_{0}^{\infty} \tf(x) \ff^{\mathrm{in}}(x) \, dx \nonumber \\
			&+ \frac{1}{2} \int_{0}^{t} \int_{0}^{\infty} \int_{0}^{\infty} \Upsilon_{\tf}(x, y) \kk(x, y) \ff(\tau, x) \ff(\tau, y) \, dy \, dx \, d\tau \label{wf3:eq}
		\end{align}
		for all $ t \in (0, T) $ and $ \tf \in L^{\infty}(\mathbb{R}_+) $, where
		\begin{equation}
			\Upsilon_{\tf}(x, y) := \int_{0}^{x + y} \tf(z) \bb(z, x, y) \, dz - \tf(x) - \tf(y), \quad (x,y)\in {\mathbb{R}_+^2.} \label{upsilon:eq}
		\end{equation}
	
	\end{defn}
	
	\begin{defn}[{Mass-conserving weak solutions}]
		Let $T \in (0,\infty]$. A weak solution $\ff$ to \eqref{main:eq}-\eqref{in:eq} on $[0,T)$, according to \Cref{defn:weaksolution}, is said to be mass-conserving on $[0,T)$ if
		\begin{align*}
			\mm_1(\ff(t)) = \mm_1(\ff^{\mathrm{in}}), \qquad t \in [0,T).
		\end{align*}
	\end{defn}

	
	\begin{rem}
	It is important to verify that the \Cref{defn:weaksolution} is well-defined. Specifically, if $\tf \in L^{\infty}(\mathbb{R}_+)$ and $\bb$ satisfies \eqref{nop_assump:eq}, then
	\begin{equation}
		|\Upsilon_{\tf}(x,y)| \leq \int_0^{x+y} |\tf(z)|\bb(z,x,y)\, dz + |\tf(x)| + |\tf(y)| \leq (\gamma + 2) \|\tf\|_{L^\infty(\mathbb{R}_+)} \label{upsilon_bound:eq}
	\end{equation}
	for all $(x, y) \in \mathbb{R}_+^2$. This estimate, combined with \eqref{wf1:eq} and \eqref{wf2:eq}, confirms that all terms in \eqref{wf3:eq} are well-defined.
	\end{rem}
	
\begin{rem}
	Let {$T \in (0,\infty]$} and a weak solution $\ff$ to \eqref{main:eq}-\eqref{in:eq} on $[0,T)$. Given that $\tf \in L^\infty(\mathbb{R}_+)$ implies $\Upsilon_\tf \in L^\infty(\mathbb{R}_+^2)$ by \eqref{upsilon_bound:eq}, It is straightforward to deduce from \eqref{wf2:eq} and \eqref{wf3:eq} that
	\begin{align*}
			t \mapsto \int_0^\infty \tf(x) \ff(t,x) \, dx \in W_{loc}^{1,1}(0,T).
	\end{align*}	
Furthermore, 
	\begin{align}
		\frac{d}{dt}  \int_0^\infty \tf(x) \ff(t,x) \, dx  = \frac{1}{2} \int_0^\infty \int_0^\infty \Upsilon_\tf(x,y) \kk(x,y) \ff(t,x) \ff(t,y) \, dy \, dx \quad \text{for a.e. } t \in (0,T). \label{alternative wf3:eq}
	\end{align}
\end{rem}

	We start with the establishment of mass-conserving weak solutions and demonstrate the following result for the restricted class of {breakage kernel $\bb$.}
		\begin{thm}[{Existence: \texorpdfstring{$\alpha \in (0, \sigma_1]$}{alpha in (0, sigma1]}}]\label{thm:existence E=0}
		Assume that the collision kernel satisfies \eqref{kernel:eq} with ${\sigma_1}>0$ and breakage kernel $\bb$ satisfies  \eqref{localc:eq}, \eqref{nop_assump:eq}, \eqref{ui_1_assump:eq}, \eqref{ui_2_assump:eq} with {$\alpha \in (0,{\sigma_1}]$} and \eqref{mth nop_assump:eq}. Additionally, suppose that there exists a constant $l_{{\sigma_1}} \ge 1$ such that
		\begin{align}
			\int_0^{x + y} z^{{\sigma_1}} \bb(z, x, y) \, dz \ge l_{{\sigma_1}} (x^{{\sigma_1}} + y^{{\sigma_1}}), \quad (x,y)\in\mathbb{R}_+^2. \label{AOM<1b:eq}
		\end{align}
	For the initial condition
		\begin{align}
		 \ff^{\mathrm{in}} \in \Xi_{0,+} \cap \Xi_{1} \quad \text{with} \quad 0 <\rho:=\mm_1(\ff^{\mathrm{in}}), \label{initia data section 2:eq}
		\end{align}
there is at least one mass-conserving weak solution $ \ff $ to \eqref{main:eq}-\eqref{in:eq} on $ [0, T_{\gamma,\sigma}) $, where
		\begin{equation}
		T_{\gamma,\sigma} := \begin{cases}
			T_\star(\ff^{\text{in}}) & \text{if } \sigma \in [0,1) ~\text{and}~ \gamma>2, \\
			\infty & \text{if } \sigma \in [1, 2) ~ \text{and}~ \gamma>2,\\
			\infty & \text{if } \sigma \in [0, 2) ~\text{and}~ \gamma=2, \label{tstar:eq}
		\end{cases}
	\end{equation}
	with
	\begin{equation}
		T_\star(\ff^{\text{in}}) := \frac{\mm_0(\ff^{\text{in}})^{\sigma - 1}}{\kappa(1 - \sigma)(\gamma - 2)\rho^\sigma}. \label{finite time:eq}
	\end{equation}
	
	Moreover, the mass-conserving weak solution $ \ff $ to \eqref{main:eq}-\eqref{in:eq} satisfies the following integrability properties for $T\in(0,T_{\gamma,\sigma})$:
	\begin{enumerate}

		\item [(a)] For every $m>1$ and $t\in(0,T)$, $\ff(t)\in \Xi_m$, and there is a positive constant $C_1(m,T)$ depending only on $\kappa, \sigma_1, \sigma_2, \varsigma_m, \varkappa_m$, {$\mm_{0}(\ff^{\mathrm{in}})$} and $\rho$ such that
		\begin{align}
		\mm_m(\ff(t)) \le C_1(m,T)\left(1+t^{-1}\right)^{\frac{m-1}{\sigma_2}}. \label{thm2 a:eq}
		\end{align}
		\item [(b)] If $\ff^{\mathrm{in}}\in \Xi_m$ for every $m>1$, then there is a positive constant $C_2(m,T)$ depending only on $\kappa, \sigma_1, \sigma_2, \varsigma_m, \varkappa_m$, {$\mm_{0}(\ff^{\mathrm{in}})$} and $\rho$ such that
	\begin{align}
		\mm_m(\ff(t)) \le \max\left\{\mm_m(\ff^{\mathrm{in}}), C_2(m,T)\right\} \label{thm1 b:eq}
		\end{align}
	for	all $t\in [0,T)$.
	\end{enumerate}
\end{thm}
	\begin{rem}\label{remark: mu sigma 1}
	The power law breakage kernel $ \bb_\nu $ defined in \eqref{example:eq} satisfies \eqref{AOM<1b:eq} with
	\begin{align*}
		l_{{\sigma_1}} &= \frac{2^{{\sigma_1} - 1} (\nu + 2)}{{\sigma_1} + \nu + 1} \quad \text{for} ~ \nu \in \left(-1,\nu_{\sigma_1} \right] \subset (-1, 0],\quad  \text{where}~ \nu_{\sigma_1}:= -\frac{1 + {\sigma_1} - 2^{{\sigma_1}}}{1 - 2^{{\sigma_1} - 1}} .
	\end{align*}
		In particular, \Cref{thm:existence E=0} is valid for all $ \bb_\nu $, with $-1<\nu\le\nu_{\sigma_1} <0$. 
	\end{rem}

	\Cref{thm:existence E=0} provides the existence of weak solutions to \eqref{main:eq}-\eqref{in:eq} for the case where the transfer of mass  is allowed, which was excluded in \cite{GL22021}. {However, this result is restricted to the breakage kernel $\bb$, which satisfies \eqref{AOM<1b:eq}, while expanding the class of initial data considered in \cite{GL12021}.}  The proof of \Cref{thm:existence E=0} is established using a compactness argument in the space $ \Xi_{0,+} \cap \Xi_{1} $, a method initially developed in \cite{stewart1989} for the coagulation-fragmentation equation and later modified  for \eqref{main:eq}-\eqref{in:eq} in \cite{GL22021} for the case where the transfer of mass is not allowed. Here, due to the presence of mass transfer, the time monotonicity of superlinear moments established in \cite{GL22021} is no longer valid, which is now handled by using a lower bound on the moment of order ${\sigma_1}$, {as in \cite{ali2024}}.\\

Next, let us state another result in \Cref{thm:existence}, on the existence of mass-conserving solution to \eqref{main:eq}-\eqref{in:eq}, which enlarges the class of breakage kernels considered in \Cref{thm:existence E=0}. Additionally, this result includes the case $\sigma_1=0$. However, \Cref{thm:existence} is valid for a further restricted set of initial data. Furthermore, this result extends \cite[Remark 2.4]{GL12021} to a larger class of collision kernels and breakage kernels.
	\begin{thm}[{Existence: $\alpha\in(0,1)$}]\label{thm:existence}
	Assume that the collision kernel satisfies \eqref{kernel:eq} and breakage kernel $\bb$ satisfies  \eqref{localc:eq}, \eqref{nop_assump:eq}, \eqref{ui_1_assump:eq}, \eqref{ui_2_assump:eq} with $\alpha \in (0,1)$ and \eqref{mth nop_assump:eq}. In addition, suppose that there exists $ L_{-\alpha} \ge 2^{-\alpha} $ such that
		\begin{align}
			\int_0^{x + y} z^{-\alpha} \bb(z, x, y) \, dz \le \frac{L_{-\alpha}}{2} (x^{-\alpha} + y^{-\alpha}), \quad (x, y) \in \mathbb{R}_+^2. \label{neg_nop_assump:eq}
		\end{align}
		For initial data satisfying \eqref{initia data section 2:eq} and
		\begin{align}
\int_0^\infty x^{-\alpha}\ff^{\mathrm{in}}(x)dx<\infty, \label{thm negative in:eq}
		\end{align}
		there is at least one mass-conserving weak solution $ \ff $ to \eqref{main:eq}-\eqref{in:eq} on $ [0, T_{\gamma,\sigma})$, where $T_{\gamma,\sigma}$ is defined in \eqref{tstar:eq}. Moreover, the mass-conserving weak solution $ \ff $ to \eqref{main:eq}-\eqref{in:eq} satisfies the following integrability properties for $T\in(0,T_{\gamma,\sigma})$:
	\begin{enumerate}
				\item [(a)] For every $m > 1$ and $t \in (0, T)$, we have $\ff(t) \in \Xi_m$. Additionally, there exists a positive constant $D_1(m, T)$, depending only on $\kappa, \sigma_1, \sigma_2, \varsigma_m, \varkappa_m, \rho, \mm_{0}(\ff^{\mathrm{in}})$, and $\mm_{-\alpha}(\ff^{\mathrm{in}})$ such that
		\begin{align}
			\mm_m(\ff(t)) \le D_1(m,T)\left(1+t^{-1}\right)^{\frac{m-1}{\sigma_2}}. \label{thm2 a:eq}
		\end{align}
		\item [(b)]If $\ff^{\mathrm{in}} \in \Xi_m$ for every $m > 1$, then there exists a positive constant $D_2(m, T)$, depending only on $\kappa, \sigma_1, \sigma_2, \varsigma_m, \varkappa_m, \rho, \mm_{0}(\ff^{\mathrm{in}})$, and $\mm_{-\alpha}(\ff^{\mathrm{in}})$, such that
		\begin{align}
			\mm_m(\ff(t)) \le \max\left\{\mm_m(\ff^{\mathrm{in}}), D_2(m,T)\right\} \label{thm2 b:eq}
		\end{align}
		for	all $t\in [0,T)$.
	\end{enumerate}
		
	\end{thm}
	\begin{rem}\label{rem:neg}
The power law breakage kernel $ \bb_\nu $ defined in \eqref{example:eq} satisfies \eqref{neg_nop_assump:eq} with
	\begin{align*}
		L_{-\alpha} = \frac{\nu + 2}{\nu + 1 - \alpha} 	\quad	 \text{for} ~ {\nu \in (\alpha-1, 0]}.
	\end{align*}
In particular, \Cref{thm:existence} is valid for all $ \bb_\nu $, with {$\alpha-1<\nu\le0$.}
	\end{rem}	
	
	\begin{rem}
		For the nonlinear fragmentation equation without mass transfer, the existence result established in \cite[Theorem 1.4]{GL22021} includes the case of classical multiplicative collision kernel corresponding to $\sigma_1 = \sigma_2 = 1$. However, the existence results established in \Cref{thm:existence E=0} and \Cref{thm:existence} for \eqref{main:eq}-\eqref{in:eq} with mass transfer do not cover this case, and it remains uncertain whether these theorems can be extended to this scenario. 
	\end{rem}

The next step toward establishing the well-posedness of \eqref{main:eq} within an appropriate framework involves addressing the issue of uniqueness. Notably, the proof follows a similar approach to the uniqueness results derived for \eqref{main:eq} without mass transfer; see \cite[Proposition 1.6]{GL22021}, which requires the finiteness of moment of order $1+\sigma_2$.

\begin{thm}\label{thm:uniquesness}
	Suppose there exists at least one mass-conserving weak solution $\ff$ to \eqref{main:eq} as defined in \Cref{defn:weaksolution}, provided by either \Cref{thm:existence E=0} or \Cref{thm:existence}. {If $\ff^{\mathrm{in}}\in \Xi_{1+{\sigma_2}}$,} then there exists a unique weak solution $\ff$ to \eqref{main:eq} on the interval $[0,T_{\gamma,\sigma})$ such that, for each $t\in [0,T_{\gamma,\sigma})$, $\mm_{1+\sigma_2}(\ff)\in L^1(0,t)$.
\end{thm}

\begin{proof}
	Refer to \cite[Proposition 1.6]{GL22021}.
\end{proof}

	\section{Existence via Compactness Method}\label{sec:existence}
	This section focuses on the establishment of the proofs for \Cref{thm:existence E=0}, and  \Cref{thm:existence}. Throughout this section, we assume that assumptions \eqref{localc:eq}, \eqref{nop_assump:eq}, \eqref{kernel:eq}, \eqref{ui_1_assump:eq}, \eqref{ui_2_assump:eq} and \eqref{mth nop_assump:eq} on $\kk$ and $\bb$ are satisfied and the constants $\Pi_i$ for {$i\in \{1,2,\dots, 12 \}$} are positive and depend only on $\kk$, $\bb$, and $\ff^{\mathrm{in}}$, with any dependence on additional parameters being indicated explicitly. Let 
	 \begin{align}
	 \ff^{\mathrm{in}} \in \Xi_{0,+} \cap \Xi_1 \label{in_0 and 1:eq}
	 \end{align}
	be an initial condition such that 
	\begin{align}
	\rho = \mm_1(\ff^{\mathrm{in}}) > 0. \label{rho :eq}
	\end{align}
		We begin by demonstrating the well-posedness of \eqref{main:eq}-\eqref{in:eq} through an appropriately constructed estimate for the collision kernel $ \kk $. Specifically, for integer values of $ n \geq 1 $, the truncated collision kernel $ \kk_n $ and the corresponding initial data $ \ff_n^{\mathrm{in}} $ are defined as follows
	
	\begin{equation}
		\kk_n(x, y) := \kk(x, y)\textbf{1}_{(0,n)}(x+y),\quad  (x, y) \in \mathbb{R}_+^2. \label{tkernel:eq}
	\end{equation}
	
		\begin{equation}
		\ff_n^{\mathrm{in}} (x):= \ff^{\mathrm{in}}\textbf{1}_{(0,n)}(x), \quad x\in \mathbb{R}_+, \label{tin:eq}
	\end{equation}
	
	\begin{prop}\label{prop:locwp}
		For each $ n \geq 1 $, there exists a unique non-negative strong solution
		\begin{equation}
			\ff_n \in \mathcal{C}^{1}([0,T_n), L^1((0,n),dx))
		\end{equation}
		to the equation
		\begin{align}
			\partial_t \ff_n(t,x) = & \frac{1}{2} \int_{x}^{n} \int_{0}^{y} \bb(x,y-z,z) \kk_n(y-z,z) \ff_n(t,y-z) \ff_n(t,z) \, dz \, dy \nonumber \\
			& - \int_{0}^{n-x} \kk_n(x,y) \ff_n(t,x) \ff_n(t,y) \, dy, \quad (t,x) \in \mathbb{R}_+^2, \label{tmain:eq}
		\end{align}
		with the initial condition
		\begin{equation}
			\ff_n(0,x) = \ff_n^{\mathrm{in}}(x) \geq 0, \quad x \in \mathbb{R}_+, \label{t_initialdata:eq}
		\end{equation}
		defined for a maximal existence time $ T_n \in (0, \infty] $. Additionally, if $ T_n < \infty $, then
	{	\begin{equation}
			\lim_{t\to T_n} \int_0^n \ff_n(t,x)dx= \infty. \label{t_blow_up:eq}
		\end{equation}}
		
		Moreover, $ \ff_n $ satisfies the truncated mass conservation
		\begin{equation}
			\int_0^n x\ff_n(t,x)dx = \int_0^nx\ff_n^{\mathrm{in}}(x)dx \le \rho, \quad t \in [0,T_n). \label{tmassconservation:eq}
		\end{equation}
	\end{prop}
	
	\begin{proof}
		The existence and uniqueness of a non-negative strong solution \( \ff_n \in \mathcal{C}^1([0, T_n), L^1(0, n)) \) to \eqref{tmain:eq}-\eqref{tin:eq}, along with the blowup property \eqref{t_blow_up:eq} and the truncated mass conservation, follow directly from \cite[Proposition 1]{barik2020} or \cite[Section 2]{walker2002}.
	\end{proof}
	
	To extend $ \ff_n $ to $ (0, T) \times \mathbb{R}_+ $, we set $ \ff_n(t, x) = 0 $ for $ (t, x) \in (0, T) \times (n, \infty) $ and still denote this extension as $ \ff_n $. Then, from \eqref{tkernel:eq}, \eqref{tmain:eq}, \eqref{tin:eq}, and \Cref{prop:locwp}, it follows that
	\begin{align}
		\frac{d}{dt} \int_0^\infty \tf(x) \ff_n(t,x) \, dx &= \frac{1}{2} \int_0^\infty \int_0^\infty \Upsilon_{\tf}(x,y) \kk_n(x,y) \ff_n(t,x) \ff_n(t,y) \, dy \, dx \label{twf:eq}
	\end{align}
	for all $ t \in (0,T_n) $ and $ \tf \in L^\infty(\mathbb{R}_+) $.
	

		\begin{lem}\label{lem:UBOZM}
			Let $T\in (0,T_{\gamma,\sigma})$, where $T_{\gamma,\sigma}$ being defined in \eqref{tstar:eq}, then there exists a positive constant $\Pi_1(T)$ such that
				\begin{equation}
					\mm_0(\ff_n(t)) \le \Pi_1(T)\,, \label{UBZM:eq}
				\end{equation}
				for all $t\in[0,T]$ and $n\ge1$.
		\end{lem}
	\begin{proof}
		Let $t \in [0,T]$ and set $\tf(x)=1$ for $x\in \mathbb{R}_+$, then from \eqref{nop_assump:eq}, \eqref{upsilon:eq} and \eqref{twf:eq}, we deduce that
				\begin{align}
				\frac{d}{dt} \mm_0(\ff_n(t)) = \frac{1}{2} \int_0^\infty \int_0^\infty \left[N(x,y) - 2 \right] \kk_n(x,y) \ff_n(t,x) \ff_n(t,y)\ dydx.\label{UBZM 0:eq}
			\end{align}
				On the one hand, when $\gamma=2$ (recalling that $\gamma$ is defined in \eqref{nop_assump:eq}), from \eqref{UBZM 0:eq}, we get for $\sigma\in[0,2]$
			\begin{align}
				\frac{d}{dt} \mm_0(\ff_n(t)) = 0 \quad \implies \quad \mm_0(\ff_n(t))= \mm_0(\ff_n^{\mathrm{in}})\le  \mm_0(\ff^{\mathrm{in}}). \label{UBZM 3:eq}
			\end{align}
		On the other hand, when $\gamma \neq 2$, from  \eqref{nop_assump:eq}, \eqref{kernel:eq}, \eqref{twf:eq} and \eqref{UBZM 0:eq}, we deduce that
			
		\begin{align}
			\frac{d}{dt} \mm_0(\ff_n(t)) & \le \frac{\gamma-2}{2} \int_0^\infty \int_0^\infty \kk(x,y) \ff_n(t,x) \ff_n(t,y)\ dydx \nonumber\\
			&\le \kappa(\gamma-2)\mm_{{\sigma_1}}(\ff_n(t))\mm_{{{\sigma_2}}}(\ff_n(t)). \label{UBZM 1:eq}
		\end{align}
		
	Since $0\le {\sigma_1}\le{\sigma_2} \le 1$ by \eqref{kernel:eq}, it follows from H\"older's inequality that
	\begin{align}
	 	\mm_{{\sigma_1}}(\ff_n(t))&\le \mm_0(\ff_n(t))^{1-{{\sigma_1}}}\mm_1(\ff_n(t))^{{\sigma_1}}, \label{lambda_1:eq}\\
	 \mm_{{\sigma_2}}(\ff_n(t))&\le \mm_0(\ff_n(t))^{1-{{\sigma_2}}}\mm_1(\ff_n(t))^{{\sigma_2}}. \label{lambda_2:eq}
	\end{align}
		Recall that $\sigma={\sigma_1}+{\sigma_2}$ and use \eqref{tmassconservation:eq}, \eqref{lambda_1:eq} and \eqref{lambda_2:eq} in \eqref{UBZM 1:eq} to get
			\begin{align}
			\frac{d}{dt} \mm_0(\ff_n(t)) &  \le \kappa(\gamma-2)\mm_0(\ff_n(t))^{2-{\sigma}}\mm_1(\ff_n(t))^{\sigma} \nonumber\\
			&  \le \kappa(\gamma-2)\rho^{\sigma}\mm_0(\ff_n(t))^{2-{\sigma}}. \label{UBZM 2:eq}
		\end{align}

Since the bound on $\mm_0(\ff_n(t))$ to be derived from \eqref{UBZM 2:eq} depends on $\sigma$, it is necessary to analyze  \eqref{UBZM 2:eq} for  different ranges of $\sigma$.
		\begin{enumerate}
			\item If $\sigma \in  [0,1)$, then integrating \eqref{UBZM 2:eq} over time, we obtain
			\begin{equation*}
			{\mm_0(\ff_n(t)) }\le \left( \mm_0(\ff_n^{\mathrm{in}})^{\sigma-1} - \kappa(1-\sigma)(\gamma-2) \rho^\sigma t \right)^{-1/(1-\sigma)}\,
			\end{equation*}
			for $t\in [0,T),$ where $T\in \left(0, {T_\star(n)} \right)$ {with $T_\star(n)$ defined as} $${T_\star(n)}:=\frac{\mm_0(\ff_n^{\mathrm{in}})^{\sigma-1}}{\kappa(1-\sigma)(\gamma-2) \rho^\sigma} \ge T_\star(\ff^{\mathrm{in}}).$$
			Since $\sigma\in [0,1)$ and $\mm_0(\ff_n^{\mathrm{in}})^{\sigma-1} \ge \mm_0(\ff^{\mathrm{in}})^{\sigma-1}$, we get
				\begin{equation*}
					{\mm_0(\ff_n(t)) } \le \left( \mm_0(\ff^{\mathrm{in}})^{\sigma-1} - \kappa(1-\sigma)(\gamma-2) \rho^\sigma T \right)^{-1/(1-\sigma)}\,.
			\end{equation*}
			
			\item If $\sigma= 1$,  then 			
				\begin{align*}
				 \mm_0(\ff_n(t))  \le \mm_0(\ff_n^{\mathrm{in}})e^{\kappa(\gamma-2) \rho t}\le \mm_0(\ff^{\mathrm{in}})e^{{\kappa}(\gamma-2) \rho T}.
			\end{align*}
			
			\item If $\sigma \in (1,2]$,  then 
				\begin{equation*}
				\mm_0(\ff_n(t)) \le \left( \mm_0(\ff_n^{\mathrm{in}})^{\sigma-1} + \kappa(\sigma-1)(\gamma-2) \rho^\sigma t \right)^{1/(\sigma-1)}\,,
			\end{equation*}
			for $t\in [0,T]$. Since $\sigma\in (1,2]$ and $\mm_0(\ff_n^{\mathrm{in}})^{\sigma-1} \le \mm_0(\ff^{\mathrm{in}})^{\sigma-1}$, we get
			\begin{equation*}
				\mm_0(\ff_n(t)) \le \left( \mm_0(\ff^{\mathrm{in}})^{\sigma-1} + \kappa(\sigma-1)(\gamma-2) \rho^\sigma T \right)^{1/(\sigma-1)}\,.
			\end{equation*}
		\end{enumerate}
	\end{proof}

\subsection{Existence:~$\alpha\in(0,\sigma_1]$}
In this section, we additionally assume that ${\sigma_1}>0$ and {$\alpha\in(0,{\sigma_1}]$}, and the breakage kernel $\bb$ satisfies \eqref{AOM<1b:eq}. The first step is to establish bounds on the superlinear moments, which relies on obtaining a lower bound for the moment of order $\sigma_1$ { by following the approach in \cite{ali2024}}. The next lemma focuses on deriving this lower bound, being inspired by \cite[Lemma 3.3]{ali2024}.
			\begin{lem}\label{lem: uniform lower bound zero moment}
			For any $T>0$, 
			\begin{align}
				\mm_{{{\sigma_1}}}(\ff_n(t)) \ge\Pi_2 \label{a lower bound on lambda_1 moment:eq}
			\end{align}
			for all $t\in [0,T]$ and $n\ge n_0$, where $\Pi_{2}:=\frac{\mm_{{\sigma_1}}(\ff^{\mathrm{in}})}{2}.$
		\end{lem}
		
		\begin{proof}Let $t \in [0,T]$ and $ n_0 \ge 1 $ be sufficiently large so that
			\begin{align}
				\int_{n_0}^\infty x^{{\sigma_1}}\ff^{\mathrm{in}}(x)dx\le \frac{\mm_{{\sigma_1}}(\ff^{\mathrm{in}})}{2}. \label{LB lambda 1:eq}
			\end{align}
			For $t \ge 0$, we set $\tf(x) = x^{\sigma_1}$ for $x\in \mathbb{R}_+$. Then, from \eqref{AOM<1b:eq}, we obtain the estimate for $\Upsilon_{\tf}$ as
			\begin{align*}
				\Upsilon_{\tf}(x,y)&= \int_0^{x+y}z^{{\sigma_1}}\bb(z,x,y)dz-x^{{\sigma_1}}-y^{{\sigma_1}}\\
				&\ge (l_{{\sigma_1}}-1)(x^{{\sigma_1}}+y^{{\sigma_1}}) \ge 0.
			\end{align*}
			
			Using the above estimate along with \eqref{kernel:eq}, \eqref{tkernel:eq} in \eqref{twf:eq}, we get
			\begin{align*}
				\frac{d}{dt}\mm_{{\sigma_1}}(\ff_n(t)) \ge0.
			\end{align*}
			Integrating with respect to time, and using \eqref{LB lambda 1:eq}, we obtain
			\begin{align*}
				\mm_{{\sigma_1}}(\ff_n(t))\ge \mm_{{\sigma_1}}(\ff_n^{\mathrm{in}})\ge \frac{\mm_{{\sigma_1}}(\ff^{\mathrm{in}})}{2}
			\end{align*}
			for $t\ge0$. Thus, the desired result follows.
		\end{proof}

{Next, we examine an estimate for superlinear moments as in \cite[Lemma 3.4]{ali2024} and \cite[Lemma 8.2.41.]{bll2019} with the help of \Cref{lem: uniform lower bound zero moment}.}
		
		\begin{lem} \label{lem:supermoments 1} Let $T\in (0,T_{\gamma,\sigma})$, where $T_{\gamma,\sigma}$ being defined in \eqref{tstar:eq}. Then for $m > 1$ and $T \in (0,T_{\gamma,\sigma})$, there exists a positive constant $\Pi_3(m,T)$ such that
			\begin{align}
				\mm_{m}(\ff_n(t)) \le  \Pi_3(m,T)\left(1+t^{-1}\right)^{\frac{m-1}{{{\sigma_2}}}} \label{SLMB1: eq}
			\end{align}
			for all $t\in[0,T]$ and $n\ge n_0$. In addition, if $\ff^{\mathrm{in}}\in \Xi_m$, there exists a positive constant $\Pi_4(m,T)$ such that
				\begin{align}
				\mm_{m}(\ff_n(t)) \le  \max\{\mm_{m}(\ff^{\mathrm{in}}), \Pi_4(m,T)\} \label{SLMB2:eq}
			\end{align}
				for all $t\in[0,T]$ and $n\ge n_0$.
		\end{lem}
		
		\begin{proof}
Let $t \in [0,T]$, $m>1$ and $\tf( x)= x^m$ for $ x\in\mathbb{R}_+$. Then, from \eqref{kernel:eq}, \eqref{mth nop_assump:eq}, \eqref{tkernel:eq} and \eqref{twf:eq}, we get
			\begin{align}
				\frac{d}{dt}\mm_m(\ff_n(t))\le & \int_{0}^{\infty}\int_{0}^{\infty}[\varsigma_m x y^{m-1}-\varkappa_m x^m] \kk_n( x, y)\ff_n (t, x)\ff_n (t, y)d y d x \nonumber\\
				\le& \kappa\varsigma_m[\mm_{1+{{\sigma_1}}}(\ff_n(t))\mm_{m+{{\sigma_2}}-1}(\ff_n(t))+\mm_{1+{{\sigma_2}}}(\ff_n(t))\mm_{m+{{\sigma_1}}-1}(\ff_n(t))] \nonumber \\ &-\kappa \varkappa_m\mm_{m+{{\sigma_2}}}(\ff_n(t))\mm_{{{\sigma_1}}}(\ff_n(t)). \label{SLMB me:eq}
			\end{align}
		We now proceed to estimate the upper bounds of the following moments using \eqref{tmassconservation:eq}, \eqref{UBZM:eq} and H\"older's inequality. Specifically, we obtain
			\begin{align*}
				\mm_{1+{{\sigma_1}}}(\ff_n(t))&\le \rho^{\frac{m+{{\sigma_2}}-1-{{\sigma_1}}}{m+{{\sigma_2}}-1}}\mm_{m+{{\sigma_2}}} (\ff_n(t))^{\frac{{{\sigma_1}}}{m+{{\sigma_2}}-1}},\\
				\mm_{m+{{\sigma_2}}-1}(\ff_n(t))&\le \mm_{{{\sigma_2}}}(\ff_n(t))^{\frac{1}{m}}\mm_{m+{{\sigma_2}}}(\ff_n(t))^{\frac{m-1}{m}}\\
				&\le\Pi_1(T)^{\frac{1-{{\sigma_2}}}{m}}\rho^{\frac{{\sigma_2}}{m}}\mm_{m+{{\sigma_2}}}(\ff_n(t))^{\frac{m-1}{m}},\\
				\mm_{1+{{\sigma_2}}}(\ff_n(t))&\le \rho^{\frac{m-1}{m+{{\sigma_2}}-1}}\mm_{m+{{\sigma_2}}}(\ff_n(t))^{\frac{{{\sigma_2}}}{m+{{\sigma_2}}-1}},\\
				\mm_{m+{{\sigma_1}}-1}(\ff_n(t))&\le \mm_{{{\sigma_1}}}(\ff_n(t))^{\frac{1+{{\sigma_2}}-{{\sigma_1}}}{m+{{\sigma_2}}-{{\sigma_1}}}}\mm_{m+{{\sigma_2}}}(\ff_n(t))^{\frac{m-1}{m+{{\sigma_2}}-{{\sigma_1}}}}\\
				&\le \Pi_1(T)^{\frac{(1-{\sigma_1})(1+{{\sigma_2}}-{{\sigma_1}})}{m+{{\sigma_2}}-{{\sigma_1}}}}\rho^{\frac{{\sigma_1}(1+{{\sigma_2}}-{{\sigma_1}})}{m+{{\sigma_2}}-{{\sigma_1}}}}\mm_{m+{{\sigma_2}}}(\ff_n(t))^{\frac{m-1}{m+{{\sigma_2}}-{{\sigma_1}}}}.
				
			\end{align*}
			
			Finally, we obtain for some positive constant $\Pi_5(m,T)$
			\begin{align}
				\mm_{1+{{\sigma_1}}}(\ff_n(t)) \mm_{m+{{\sigma_2}}-1}(\ff_n(t))  \le& \Pi_5(m,T) \mm_{m+{{\sigma_2}}}(\ff_n(t))^{\gamma_1},\\
				\mm_{1+{{\sigma_2}}}(\ff_n(t)) \mm_{m+{{\sigma_1}}-1}(\ff_n(t))  \le& \Pi_5(m,T) \mm_{m+{{\sigma_2}}}(\ff_n(t))^{\gamma_2},
			\end{align}
			where 
			\begin{align*}
				\gamma_1&= \frac{(m-1)(m-1+{{\sigma_2}})+m{{\sigma_1}}}{m(m+{{\sigma_2}}-1)}>0,\\
				\gamma_2 &=\frac{(m+{{\sigma_2}}-1)^2 +{{\sigma_2}}(1-{{\sigma_1}})}{(m+{{\sigma_2}}-1)(m+{{\sigma_2}}-{{\sigma_1}})}>0.
			\end{align*}
			We note that, since $m>1$, ${{\sigma_2}}\ge {{\sigma_1}}$ and ${{\sigma_1}}<1$ 
			\begin{align*}
				1-\gamma_1 =& \frac{m+{{\sigma_2}}-1-m{{\sigma_1}}}{ m(m+{{\sigma_2}}-1)}\ge \frac{m(1-{{\sigma_1}})+ {{\sigma_1}}-1}{m(m+{{\sigma_2}}-1)}= \frac{(m-1)(1-{{\sigma_1}})}{m(m-{{\sigma_2}}+1)}>0,\\
				1-\gamma_2=& \frac{(m-1)(1-{{\sigma_1}})}{(m+{{\sigma_2}}-1)(m+{{\sigma_2}}-{{\sigma_1}})}>0.
			\end{align*}
			With the help of \eqref{a lower bound on lambda_1 moment:eq}, \eqref{SLMB me:eq} and above estimates, we obtain
			\begin{equation*}
				\frac{d}{dt} \mm_m(\ff_n(t))\le  \kappa \Pi_5(m,T)[\mm_{m+{{\sigma_2}}}(\ff_n(t))^{\gamma_1} +\mm_{m+{{\sigma_2}}}(\ff_n(t))^{\gamma_2}]-\kappa \varkappa_m \Pi_2\mm_{m+{{\sigma_2}}}(\ff_n(t)).
			\end{equation*}
			Since $(\gamma_1,\gamma_2)\in (0,1)^2$, we deduce from Young's inequality that there is $\Pi_6(m,T)$ such that
			\begin{equation}
				\frac{d}{dt} \mm_m(\ff_n(t))\le \Pi_6(m,T) -\frac{1}{2}\kappa \varkappa_m \Pi_2\mm_{m+{{\sigma_2}}}(\ff_n(t)). \label{SLMB3:eq}
			\end{equation}
			Since $m \in [1,m+{{\sigma_2}}]$, we use  H\"{o}lder's inequality to show that
			\begin{equation*}
				\mm_m(\ff_n(t)) \le \mm_1(\ff_n(t))^{\frac{{{\sigma_2}}}{m+{{\sigma_2}}-1}} \mm_{m+{{\sigma_2}}}(\ff_n(t))^{\frac{m-1}{m+{{\sigma_2}}-1}} \le \rho^{\frac{{{\sigma_2}}}{m+{{\sigma_2}}-1}} \mm_{m+{{\sigma_2}}}(\ff_n(t))^{\frac{m-1}{m+{{\sigma_2}}-1}}.
			\end{equation*}  
			Hence ,
			\begin{equation}
				\rho^{-\frac{{{\sigma_2}}}{m-1}} \mm_m(\ff_n(t))^{\frac{m+{{\sigma_2}}-1}{m-1}} \le  \mm_{m+{{\sigma_2}}}(\ff_n(t)).
			\end{equation}
			
						Using the above estimate in \eqref{SLMB3:eq}, we get for some positive constant {$\Pi_7(m)$} such that
					    \begin{align}
								\frac{d}{dt}\mm_m(\ff_n(t)) +{\Pi_7(m)} \mm_{m}(\ff_n(t))^{\frac{m+{{\sigma_2}}-1}{m-1}}\le \Pi_{6}(m,T), \qquad {t\in(0,T)}, \label{SLMB4:eq}
					     \end{align}
			where $${\Pi_7(m):=\frac{1}{2}\kappa \varkappa_m\Pi_2\rho^{\frac{-\sigma_2}{m-1}}.}$$

			Introducing 
			\begin{equation*}
				X(t) = \Big(R_1 + R_2 t^{-1}\Big)^{\frac{m-1}{{{\sigma_2}}}}, \qquad t>0,
			\end{equation*}
			with 
			\begin{equation*} 
				R_2 = \frac{m}{{{\sigma_2}} {\Pi_7(m)}}, \qquad R_1 = \Bigg( \frac{m\Pi_6(m,T)}{{\Pi_7(m)}} \Bigg)^{\frac{{{\sigma_2}}}{m+{{\sigma_2}}-1}}.
			\end{equation*}
		A straightforward calculation reveals that $X$ serves as a supersolution to \eqref{SLMB4:eq} such that 
		$$\lim_{t\to 0}X(t)=+\infty.$$
		 Consequently, the comparison principle indicates that $\mm_m(\ff_n(t)) \leq X(t)$ for $t \in(0,T)$, which establishes
			\begin{equation}
				\mm_m(\ff_n(t)) \le \Pi_3(m,T) \Big(1+t^{-1}\Big)^{\frac{m-1}{{{\sigma_2}}}}, \qquad t\in(0,T),
			\end{equation}
			for $m>1$  with $\Pi_3(m,T)= \max\big\{R_1, R_2\big\}^{\frac{m-1}{{{\sigma_2}}}}.$
			We also observe that \eqref{tin:eq}, \eqref{SLMB4:eq} and the comparison principle imply \eqref{SLMB2:eq}.
			
		\end{proof}

Note that the estimate \eqref{SLMB1: eq} established in \Cref{lem:supermoments 1} exhibits a singularity at $t = 0$, while the estimate \eqref{SLMB2:eq} avoids such singularities but require more integrability on $\ff^{\mathrm{in}}$, i.e. $\ff^{\mathrm{in}}\in \Xi_m$. To address these limitations, we derive a more refined estimate for superlinear moments {as in \cite[Lemma 3.6]{ali2024} and \cite[Lemma 8.2.42]{bll2019}} by applying a variant of de la Vall\'ee Poussin theorem (see \cite[Theorem 2.8]{phl2014}) in conjunction with the estimate \eqref{SLMB1: eq} from \Cref{lem:supermoments 1}. Since $\ff^{\mathrm{in}}$ satisfies \eqref{in_0 and 1:eq}, the de la Vall\'ee Poussin theorem ensures the existence of a convex function $\psi \in \mathcal{C}^2([0, \infty))$ such that $\psi'$ is a positive concave function with $\psi(0) = \psi'(0) = 0$.

\begin{align}
	\mm_{\psi}(\ff^{\mathrm{in}}):= \int_0^\infty \psi(x) \ff^{\mathrm{in}}(x) \, dx < \infty.
\end{align}
In addition, $\psi$ satisfies
\begin{align}
		( x+ y) \left[ \psi( x+ y) -\psi( x) - \psi( y) \right] \le 2 \left[  x \psi( y) +  y \psi( x) \right]\,, \qquad ( x, y)\in \mathbb{R}_+^2, \label{sm3}
\end{align}
and a function $\psi_1$ defined as,
\begin{align}
 \psi_1(0) = 0,   \quad 		\text{and} \quad \psi_1( x) :=\frac{\psi( x)}{x} \quad \text{for}~x\in\mathbb{R}_+, \label{monotonicity:eq}
\end{align}
is concave and non-decreasing (see \cite[Proposition 2.14]{phl2014}).

	\begin{lem} \label{lem:refine supermoments}
	For any $ T \in (0, T_{\gamma,\sigma}) $, where $ T_{\gamma,\sigma} $ is defined in \eqref{tstar:eq}, there exists a constant $ \Pi_8(T) > 0 $ such that
		\begin{align}
			\mm_{\psi}(\ff_n(t)) \leq \Pi_8(T), \quad t \in [0,T], ~~n\ge n_0. \label{phi bound:eq}
		\end{align}
	\end{lem}
	
		\begin{proof} First, applying Jensen’s inequality along with \eqref{localc:eq}, \eqref{monotonicity:eq}, and the fact that $\psi_1$ is a non-decreasing concave function, we obtain
			\begin{align*}
				\int_0^{ x+ y} \psi( z) \bb( z, x, y)\ d z & = ( x+ y) \int_0^{ x+ y} \psi_1( z) \frac{ z \bb( z, x, y)}{ x+ y}\ d z \\
				& \le ( x+ y) \psi_1 \left( \int_0^{ x+ y}  z \frac{ z \bb( z, x, y)}{ x+ y}\ d z \right) \\
				& \le ( x+ y) \psi_1 \left( \int_0^{ x+ y}  z \bb( z, x, y)\ d z \right) \\
				& = ( x+ y) \psi_1( x+ y) = \psi( x+ y)\,,
			\end{align*}
		for $( x ,  y ) \in \mathbb{R}_+^2$. As a result, we have
			\begin{align}
				\Upsilon_\psi( x, y) & \le \psi( x+ y) -\psi( x) - \psi( y). \label{sm1}
			\end{align}

Let $t \in [0,T]$. It follows from \eqref{kernel:eq}, \eqref{sm3} and  \eqref{sm1} that
			
			\begin{align*}
				\frac{d}{dt} \mm_\psi(\ff_n(t))  \le& \frac{1}{2} \int_0^\infty \int_0^\infty \left[ \psi(x+y)-\psi(x) -\psi( y) \right] \kk( x, y) \ff_n(t, x) \ff_n(t, y) d y d x \\
				\le&
				2\kappa \int_0^\infty \int_0^\infty\bigg(\frac{2 x}{ x+ y}\psi( y)+\frac{2 y}{ x+ y}\psi( x)\bigg) x^{{\sigma_1}} y^{{\sigma_2}}\ff_n(t, x) \ff_n(t, y) d y d x\\
				\le& 
				4\kappa \int_0^\infty \int_0^\infty\bigg[\frac{ x^{1+{{\sigma_1}}}}{( x+ y)^{1-{{\sigma_2}}}}\psi(y)+\frac{ y^{1+{{\sigma_2}}}}{( x+ y)^{1-{{\sigma_1}}}}\psi(x)
				\bigg]\ff_n(t, x) \ff_n(t, y) d y d x\\
				\le& 
				4\kappa \int_0^\infty \int_0^\infty\left[x^\sigma \psi(y)+y^\sigma \psi(x)
				\right]\ff_n(t, x) \ff_n(t, y) d y d x.
			\end{align*}
			Finally, we obtain
			\begin{align*}
				\frac{d}{dt} \mm_\psi(\ff_n(t))  \le& 8\kappa \mm_\sigma(\ff_n(t))\mm_\psi(\ff_n(t)).
			\end{align*}
		Integrating with respect to time, we get
			\begin{align}
				\mm_\psi(\ff_n(t))\le \mm_\psi(\ff_n^{\mathrm{in}})\exp \left(8\kappa \int_0^t \mm_\sigma(\ff_n(s))ds\right). \label{convex proof 1}
			\end{align}

				Since, the estimate obtained in \eqref{convex proof 1} depends on the range of $\sigma$, we will address this in two separate cases:
			\begin{enumerate}
				\item If $\sigma \in  [0,1]$, then by H\"older's inequality 
				\begin{align}
				\int_0^t\mm_\sigma(\ff_n(s))ds \le \Pi_1(T)^{1-\sigma}\rho^{\sigma}T. \label{refine supermoment 1:eq}
				\end{align}
				
				\item If $\sigma \in (1,1+\sigma_2)$, then by using \Cref{lem:supermoments 1} with $m=\sigma<1+\sigma_2$, we get
				\begin{align}
				\int_0^t\mm_\sigma(\ff_n(s))ds &\le \Pi_3(\sigma, T)\int_0^t  \left(1+s^{-1}\right)^{\frac{\sigma-1}{{{\sigma_2}}}}ds\le \Pi_3(\sigma, T)\int_0^t  \left(1+ s^{\frac{1-\sigma}{{{\sigma_2}}}}\right)ds \nonumber\\
				& \le \Pi_3(\sigma, T)\left(T+\frac{{\sigma_2}}{1-\sigma+{\sigma_2}}T^\frac{1-\sigma+{\sigma_2}}{{\sigma_2}}\right). \label{refine supermoment 2:eq}
			\end{align}
			\end{enumerate}
	Combining \eqref{convex proof 1}, \eqref{refine supermoment 1:eq} and \eqref{refine supermoment 2:eq} ends the proof of \Cref{lem:refine supermoments}.
		\end{proof}

	
	The next step is to establish the uniform integrability estimate.
		
	\begin{lem} \label{lem:ui}
		Let $\delta>0$ and  $T \in (0, T_{\gamma, \sigma}) $, where $ T_{\gamma, \sigma} $ is defined in $ \eqref{tstar:eq} $, there is $\Pi_9(T)>0$ such that
	\begin{align}
			{P_{n}(t,\delta)} \le   \Pi_9(T)\left[P^{\mathrm{in}}(\delta) +\eta(\delta)\right] \,, \quad t\in [0,T], ~n\ge1\,, \label{ui estimate:eq}
		\end{align} 
		where $P_{n}$ and $P^{\mathrm{in}}$ defined as
			\begin{equation*}
		{	P_{n}(t,\delta)} := \sup\left\{ \int_E\ff_n(t,x)\ dx\ :\ {E\subset \mathbb{R}_+}\,, \ |E|\le \delta \right\}, \quad t\in[0,T],~~~n\ge 1,
		\end{equation*} 
		and
		\begin{equation*}
			P^{\mathrm{in}}(\delta) := \sup\left\{ \int_E \ff^{\mathrm{in}}( x)\ d x\ :\ E\subset \mathbb{R}_+\,, \ |E|\le \delta \right\}\,,
		\end{equation*} 
		respectively, and $\eta$ is defined in assumption \eqref{ui_1_assump:eq}.
	\end{lem}
		
		\begin{proof}
			Let $t\in [0,T]$, $\delta>0$ and a measurable subset $E\subset (0,\infty)$ with finite measure $|E|\le \delta$ and $\tf_E(x) :=  \mathbf{1}_E(x)$ for $x\in \mathbb{R}_+$. Using \eqref{upsilon:eq} and \eqref{ui_2_assump:eq}, we obtain
			\begin{align}
				\Upsilon_{\tf_E}( x, y) \le \eta(|E|) \left(  x^{-\alpha} + y^{-\alpha} \right), \qquad  ( x, y)\in \mathbb{R}_+^2, \label{ui 23:eq}
			\end{align}
			which is holds for $\alpha\in(0,1)$. Since {$\alpha\in(0,{\sigma_1}]$}, it follows from  \eqref{tmassconservation:eq}, \eqref{twf:eq}, \eqref{UBZM:eq} and \eqref{ui 23:eq} that
			
				\begin{align}
				\frac{d}{dt} \int_0^\infty \tf_E( x) & \ff_n(t,x)\ d x   \le   \frac{\eta(\delta)}{2} \int_0^\infty \int_0^\infty \left(  x^{-\alpha} +  y^{-\alpha} \right) \kk_n( x, y) \ff_n(t, x) \ff_n(t, y)\ d y d x \nonumber\\
				\le& \kappa \eta(\delta) \int_0^\infty \int_0^\infty \left(  x^{-\alpha} +  y^{-\alpha} \right)x^{{\sigma_1}} y^{{\sigma_2}} \ff_n(t, x) \ff_n(t, y)\ d y d x \nonumber \\
				\le & \kappa \eta(\delta)[ \mm_{\sigma_1-\alpha}(\ff_n(t))\mm_{\sigma_2}(\ff_n(t))+\mm_{\sigma_2-\alpha}(\ff_n(t))\mm_{\sigma_1}(\ff_n(t))].  \label{ui me: eq}
			\end{align}
			
			Since $0\le \sigma_1\le \sigma_2\le 1$ and $0\le \sigma_1-\alpha\le \sigma_2-\alpha<1$, we obtain the following estimates using \eqref{tmassconservation:eq}, \eqref{UBZM:eq} and H\"older's inequality
			\begin{align*}
				\mm_{{\sigma_1}}(\ff_n(t))&\le \mm_0(\ff_n(t))^{1-\sigma_1}\mm_1(\ff_n(t))^{\sigma_1} \le \Pi_1(T)^{1-{{\sigma_1}}}\rho^{{\sigma_1}}, \\
				\mm_{{\sigma_2}}(\ff_n(t))&\le \mm_0(\ff_n(t))^{1-\sigma_2}\mm_1(\ff_n(t))^{\sigma_2}\le \Pi_1(T)^{1-{{\sigma_2}}}\rho^{{\sigma_2}}, \\
				\mm_{{{\sigma_1}}-\alpha}(\ff_n(t))&\le \mm_{0}(\ff_n(t))^{1+\alpha-\sigma_1}\mm_1(\ff_n(t))^{\sigma_1-\alpha} \le \Pi_1(T)^{1+\alpha-{\sigma_1}}\rho^{\sigma_1-\alpha} ,\\
				\mm_{{{\sigma_2}}-\alpha}(\ff_n(t))&\le \mm_{0}(\ff_n(t))^{1+\alpha-\sigma_2}\mm_1(\ff_n(t))^{\sigma_2-\alpha} \le \Pi_1(T)^{1+\alpha-{\sigma_2}}\rho^{\sigma_2-\alpha} .
			\end{align*}
			Using above estimates in \eqref{ui me: eq}, we get
			
				\begin{align*}
				\frac{d}{dt} \int_0^\infty \tf_E( x)\ff_n(t,x)\ d x \le 2\kappa \eta(\delta) \Pi_1(T)^{2+\alpha-\sigma}\rho^{\sigma-\alpha}.
			\end{align*}

			By integrating with respect to time and taking the supremum over all measurable subsets {$E \subset (0, \infty)$} with finite measure $|E| \leq \delta$, we obtain
			\begin{equation*}
			{	P_{n}(t,\delta)} \le {P_{n}(0,\delta)} +2t\kappa \eta(\delta)\Pi_1(T)^{2+\alpha-\sigma}\rho^{\sigma-\alpha}.
			\end{equation*}
		Note that
						\begin{equation}
								{P_{n}(0,\delta)} \le P^{\mathrm{in}}(\delta)\,. \label{uniform int initial data:eq}
					\end{equation}
			Therefore, by \eqref{uniform int initial data:eq}, 
			\begin{align*}
			{	P_{n}(t,\delta)} \le P^{\mathrm{in}}(\delta) + 2T\kappa \eta(\delta)\Pi_1(T)^{2+\alpha-\sigma}\rho^{\sigma-\alpha},
			\end{align*} 
			which completes the proof of the \Cref{lem:ui}.
		\end{proof}
		
Next, we establish the time equicontinuity of the sequence $(\ff_n)_{n \geq 1}$.
		\begin{lem} \label{lem:equicontinuity}
			For any $T\in (0,T_{\gamma,\sigma})$, where $T_{\gamma,\sigma}$ being defined in \eqref{tstar:eq} , there is {$\Pi_{10}(T)>0$} such that
			\begin{align}
				\int_0^\infty |\ff_n(t, x)-\ff_n(s, x)|dx \le \Pi_{10}(T)|t-s|, \qquad (t,s)\in[0,T]^2,~ n\ge1.
			\end{align}
		\end{lem}
		
		\begin{proof}
			We begin with the following bound on collision kernel, which is obtained from \eqref{kernel:eq}
			\begin{align}
				\kk(x,y)\le2\kappa(1+x)(1+y), \qquad (x,y)\in\mathbb{R}_+^2. \label{kernel other form:eq}
			\end{align}
		Let $t \in [0,T]$. Using \eqref{nop_assump:eq}, \eqref{tmassconservation:eq}, \eqref{UBZM:eq}, and \eqref{kernel other form:eq}, we get
			\begin{align*}
				\frac{1}{2} \int_0^\infty  \int_x^\infty &\int_0^y \bb(x,y-z,z)  \kk_n(y-z,z) \ff_n(t,y-z) \ff_n(t,z)\ dzdydx \\
				& = \frac{1}{2} \int_0^\infty \int_0^\infty \int_0^{x+y} \bb(z,x,y)  \kk_n(x,y) \ff_n(t,x) \ff_n(t,y)\ dzdydx \\
				&  \le \frac{ \gamma}{2} \int_0^\infty \int_0^\infty \kk(x,y) \ff_n(t,x) \ff_n(t,y)\ dydx \\
				& \le  \kappa \gamma \int_0^\infty \int_0^\infty (1+x)(1+y) \ff_n(t,x) \ff_n(t,y)\ dydx \\
				&  \le  \kappa \gamma \left(\Pi_1(T)+\rho\right)^2\,.
			\end{align*}
			Next, using \eqref{tmassconservation:eq}, \eqref{UBZM:eq}, and \eqref{kernel other form:eq}, we obtain
			\begin{align*}
				\int_0^\infty \int_0^\infty \kk_n(x,y) \ff_n(t,x) \ff_n(t,y)\ dydx & \le \int_0^\infty \int_0^\infty \kk(x,y) \ff_n(t,x) \ff_n(t,y)\ dydx \\
				& \le 2\kappa (\mm_0(\ff_n(t))+\rho)^2\,.
			\end{align*}
			Combining the obtained estimates and applying \eqref{twf:eq} gives
			\begin{align*}
				\|\partial_t \ff_n(t)\|_{0} &\le \kappa (\gamma+2) (\Pi_1(T)+\rho)^2\,,
			\end{align*}
		which directly implies Lemma~\ref{lem:equicontinuity}.
		\end{proof}
Finally, we provide the proof of \Cref{thm:existence E=0} using the estimates obtained so far.
	\begin{proof}[Proof of \Cref{thm:existence E=0}]
For a fixed $T\in(0,T_{\gamma,\sigma})$, we define
	\begin{equation*}
		\mathcal{A}(T):=\{\ff_n(t) : t\in [0,T], n\geq 1 \}.
	\end{equation*}
{	Consider $\delta>0$, and a measurable subset $E$ of $\mathbb{R}_+$ with finite Lebesgue measure. Then, from \Cref{lem:ui}, the modulus of uniform integrability }
		\begin{equation*}
			P(\delta) := \sup\left\{ \int_E  \ff_n(t, x)\ dx\ :\ E\subset \mathbb{R}_+\,, \ |E|\le \delta\,, \ t\in [0,T]\,, \ n\ge 1 \right\},
		\end{equation*} 
of $\mathcal{A}(T)$ in $\Xi_{0}$ satisfies
	\begin{equation}
		P(\delta) \le \Pi_9(T)\left[P^{\mathrm{in}}(\delta)+ \eta(\delta)\right]. \label{compactness 1:eq}
	\end{equation}
	Note that the integrability property of $\ff^{\mathrm{in}}$ ensure that
				\begin{equation}
						\lim_{\delta\to 0} P^{\mathrm{in}}(\delta) = 0 \,.  \label{uniform int limit delta zero:eq}
					\end{equation}
	Owing to \eqref{ui_1_assump:eq} and \eqref{uniform int limit delta zero:eq}, we may let $\delta\to 0$ in the \eqref{compactness 1:eq} to obtain
	\begin{equation}
		\lim_{\delta\to 0} P(\delta) = 0\,. \label{compactness 3:eq}
	\end{equation}
{Moreover, using \eqref{tmassconservation:eq}, we obtain
	\begin{equation*}
		\int_{J}^{\infty}\ff_n(t,x)dx \leq \frac{\rho}{J},
	\end{equation*}
for all $n \geq 1$, $t \in [0,T]$, and $J > 1$, which leads to
	\begin{equation}
		\lim_{J \to \infty} \sup_{n,t} 	\int_{J}^{\infty}\ff_n(t,x)dx=0. \label{ch 5 tail}
	\end{equation}}
Therefore, \eqref{compactness 3:eq} and \eqref{ch 5 tail} imply the sequential weak compactness of the set $ \mathcal{A}(T)$ in $ \Xi_0 $. Moreover, the set $ \mathcal{A} $ is bounded in $ \Xi_0 $ due to \Cref{lem:UBOZM}. Hence, we can deduce that $ \mathcal{A} $ is relatively sequentially weakly compact in $ \Xi_{0} $, in accordance with the Dunford-Pettis theorem \cite[Theorem 2.3]{phl2014}. By utilizing a version of the Arzelà-Ascoli theorem \cite[Theorem 1.3.2]{vi1995}, and based on \Cref{lem:equicontinuity}, we can conclude that the sequence $ (\ff_n)_{n\geq1} $ is relatively sequentially compact in $ \mathcal{C}([0,T],\Xi_{0,w}) $. Since $ T $ can be chosen arbitrarily within $ (0,T_{\gamma,\sigma}) $, we can apply a diagonal argument to extract a subsequence from $ (\ff_n)_{n\geq 1} $ (without relabeling) and identify a function $ \ff \in \mathcal{C}([0,T_{\gamma,\sigma}),\Xi_{0,w}) $ such that
	\begin{equation}
		\lim_{n \to \infty} \sup_{t\in[0,T]} \Big \lvert \int_{0}^{\infty} (\ff_n-\ff)(t,x)\tf(x)dx  \Big \rvert=0 \label{eq:convergence}
	\end{equation}
	for every $ \tf\in L^\infty(\mathbb{R}_+) $ and $ T\in (0,T_{\gamma,\sigma}) $. One immediate consequence of \eqref{eq:convergence} and  $ \ff_n \ge0$ for each $n\ge1$ is that $ \ff\ge0$ in $ \Xi_0 $ for all $ t > 0 $. Next, we extend \eqref{eq:convergence} to the weak topology of $ \Xi_{1} $, as follows
\begin{equation}
	\lim_{n \to \infty} \sup_{t\in[0,T]} \Big \lvert \int_{0}^{\infty} x(\ff_n-\ff)(t,x)\tf(x)dx  \Big \rvert=0  \label{eq:X1mm+1convergence}
\end{equation}
for all $ T\in (0,T_{\gamma,\sigma}) $ and $ \tf  \in L^\infty(\mathbb{R}_+) $. For this, let $ T\in (0,T_{\gamma,\sigma}) $. By employing the nonnegativity of $ \psi $ along with the bounds from \eqref{phi bound:eq} and \eqref{eq:convergence}, we obtain
\begin{align*}
	\int_0^J \psi(x) \ff(t,x) dx = \lim_{n\to \infty} \int_0^J \psi(x)\ff_n(t,x)dx \le \Pi_8(T), \quad \text{for all } t \in [0, T] \text{ and } J > 1,
\end{align*}
which implies that
\begin{align}
	\mm_\psi(\ff(t)) \le \liminf_{J\to\infty} \int_0^J \psi(x)\ff_n(t,x)dx \le \Pi_8(T), \quad \text{for all } t \in [0, T] \text{ and } J > 1. \label{ls:eq}
\end{align}

 Next, we consider $\tf  \in L^\infty(\mathbb{R}_+)$, and $J>1$. Then, from \eqref{monotonicity:eq}, \eqref{phi bound:eq} and \eqref{ls:eq}, we obtain for $t\in[0,T]$
	\begin{align*}
			\left| \int_0^\infty x (\ff_n-\ff)(t,x) \tf(x)\ dx \right| \le 	&\left| \int_0^J x (\ff_n-\ff)(t,x) \tf(x)\ dx \right| \\
			&+\|\tf\|_{L^\infty(\mathbb{R}_+)}\int_J^\infty x (\ff_n+\ff)(t,x)\ dx \\
			\le 	&\left| \int_0^J x (\ff_n-\ff)(t,x) \tf(x)\ dx \right| \\
			&+\frac{J}{\psi(J)}\|\tf\|_{L^\infty(\mathbb{R}_+)}\int_J^\infty \psi(x) (\ff_n+\ff)(t,x) \ dx\\
			\le 	&\left| \int_0^J x (\ff_n-\ff)(t,x) \tf(x)\ dx \right| +\frac{2J\Pi_8(T)}{\psi(J)}\|\tf\|_{L^\infty(\mathbb{R}_+)}.
	\end{align*}
	Therefore,
	
		\begin{align*}
	\sup_{t\in[0,T]}	\left| \int_0^\infty x (\ff_n-\ff)(t,x) \tf(x)\ dx \right| \le & \sup_{t\in[0,T]} \left| \int_0^J x (\ff_n-\ff)(t,x) \tf(x)\ dx \right|\\ &+\frac{2J\Pi_8(T)}{\psi(J)}\|\tf\|_{L^\infty(\mathbb{R}_+)}.
	\end{align*}
	
	Owing to \eqref{eq:convergence}, we can take the limits as $ n \to \infty $ and $ J \to \infty $ in the preceding inequality to derive \eqref{eq:X1mm+1convergence}. Also observe that the convergence demonstrated in \eqref{eq:X1mm+1convergence} indicates that $\ff \in \mathcal{C}([0,T_{\gamma,\sigma}), \Xi_{1,w})$. Additionally, a direct consequence of \eqref{tin:eq} and \eqref{tmassconservation:eq} is that
	
	\begin{equation}
		\mm_1(\ff (t))=\lim_{n \to \infty} \mm_1(\ff_n(t))=\lim_{n \to \infty} \mm_1(\ff_n^{\mathrm{in}})= {\mm_1(\ff^{\mathrm{in}})}=\rho,\qquad t\in [0,T_{\gamma,\sigma}). \label{eq:limitmass}
	\end{equation}
	
It remains to be shown that the constructed solution is a weak solution to \eqref{main:eq}-\eqref{in:eq} in the sense of \Cref{defn:weaksolution}. To this end, by using \eqref{eq:convergence}, \eqref{eq:X1mm+1convergence}, and \eqref{eq:limitmass}, we conclude that $u \in \mathcal{C}([0,T_{\gamma,\sigma}), \Xi_{0,w}) \cap L^\infty((0,T_{\gamma,\sigma}), \Xi_1)$, thus satisfying \eqref{wf1:eq}. Additionally, $\ff$ satisfies \eqref{wf2:eq} as a result of \eqref{kernel:eq}, \eqref{kernel other form:eq}, and \eqref{wf1:eq}. Finally, we verify that $u$ fulfills \eqref{wf3:eq}.

Let $t \in (0, T_{\gamma,\sigma})$ and $\tf \in L^\infty(\mathbb{R}_+)$. Starting from \eqref{t_initialdata:eq} and \eqref{eq:convergence}, we obtain
\begin{equation*}
	\lim_{n \to \infty} \int_{0}^{\infty} \tf(x)(\ff_n(t,x) - \ff_n^{\mathrm{in}}(x)) \, dx = \int_{0}^{\infty} \tf(x)(\ff(t,x) - \ff^{\mathrm{in}}(x)) \, dx.
\end{equation*}

Next, we observe that
\begin{align*}
	\lim_{n \to \infty} \frac{\Upsilon_{\tf}(x,y) \kk_n(x,y)}{(1+x)(1+y)} = \frac{\Upsilon_{\tf}(x,y) \kk(x,y)}{(1+x)(1+y)}, 
\end{align*}
for $(x,y) \in \mathbb{R}_+^2$. Furthermore, by \eqref{upsilon_bound:eq} and \eqref{kernel other form:eq}, we have
\begin{align*}
	\left\lvert \frac{\Upsilon_{\tf}(x,y) \kk_n(x,y)}{(1+x)(1+y)} \right\rvert \leq 2\kappa (\gamma + 2) \lVert \tf \rVert_{L^\infty}.
\end{align*}

Finally, it follows from \eqref{eq:convergence} and \eqref{eq:X1mm+1convergence} that 

\begin{equation*}
	[(\tau,x,y) \mapsto \ff_n(\tau,x)\ff_n(\tau,y)] \rightharpoonup  [(\tau,x,y) \mapsto  \ff(\tau,x)\ff(\tau,y)]
\end{equation*}
in $L^1((0, t) \times \mathbb{R}_+^2,(1+x)(1+y) \, dy \, dx \, d\tau)$. Therefore, it follows from \cite[Proposition~2.61]{FL2007} that
	\begin{align*}
		\lim_{n\to \infty} \frac{1}{2} \int_{0}^{t} \int_{0}^{\infty} \int_{0}^{\infty} & \Upsilon_{\tf}(x,y) \kk_n(x,y) \ff_n(\tau,x) \ff_n(\tau,y) dy dx d\tau \\
		&=\frac{1}{2} \int_{0}^{t} \int_{0}^{\infty} \int_{0}^{\infty} \Upsilon_{\tf}(x,y) \kk(x,y) \ff (\tau,x) \ff (\tau,y) dy dx d\tau.
	\end{align*}
	
We have therefore confirmed that $\ff$ fulfills all the criteria necessary to be considered a weak solution to~\eqref{main:eq}-\eqref{in:eq} on $[0, T_{\gamma,\sigma})$. Additionally, it follows from \eqref{eq:limitmass} that the constructed weak solution on $[0, T_{\gamma,\sigma})$ is mass-conserving. It remains to establish the moment estimates stated in \Cref{thm:existence E=0} (a) and (b), which can be derived from \Cref{lem:supermoments 1} and \eqref{eq:convergence}. This completes the proof of \Cref{thm:existence E=0}.
\end{proof}
\subsection{Existence:~$\alpha\in(0,1)$}
{In this section}, we plan to show another result on existence of weak solutions to \eqref{main:eq}-\eqref{in:eq} stated in \Cref{thm:existence}. For this, we additionally assume that {$\alpha\in(0,1)$ and} the breakage kernel $ \bb $ satisfies \eqref{neg_nop_assump:eq}. {In addition to \eqref{in_0 and 1:eq}, we assume that the initial condition satisfies}
\begin{align}
	\int_0^\infty x^{-\alpha} \ff^{\mathrm{in}}(x) \, dx < \infty, \label{in_ negative:eq}
\end{align}

where $\alpha$ {is defined} in \eqref{ui_2_assump:eq}. To establish the bound on the superlinear moment, as given in \Cref{lem:supermoments 1}, it is essential to obtain a lower bound on the moment of order $ \sigma_1 $. Without the assumption \eqref{AOM<1b:eq}, however, \Cref{lem: uniform lower bound zero moment} no longer holds. Using \eqref{neg_nop_assump:eq} and \eqref{in_ negative:eq}, we derive a lower bound on the moment of order $ \sigma_1 $ {as in \cite[Lemma 3.3]{ali2024}}, which we present in \Cref{lem:LB on alpha} given below.

\begin{lem}\label{lem:LB on alpha}
	Let $ T \in (0, T_{\gamma, \sigma}) $, where $ T_{\gamma, \sigma} $ is defined in \eqref{tstar:eq}. Then there exists $ \Pi_i(T) > 0 $ for $ i \in \{11, 12\} $ such that
	\begin{enumerate}
		\item[(a)] For all $ t \ge 0 $ and $ {n \ge n_*}$, we have
		\begin{align}
			\mm_0(\ff_n(t)) \ge \frac{ \mm_0(\ff^{\mathrm{in}})}{2}. \label{LBOZM:eq}
		\end{align}
		\item[(b)] For all $ t \in [0, T] $ and $ n \ge 1 $, we have
		\begin{align}
			\mm_{-\alpha}(\ff_n(t)) \le \Pi_{11}(T). \label{UBONM:eq}
		\end{align}
		\item[(c)] For all $ t \in [0, T] $ and ${ n \ge n_*} $, we have
		\begin{align}
			\mm_{\sigma_1}(\ff_n(t)) \ge \Pi_{12}(T). \label{a nonuniform lower bound on lambda_1 moment:eq}
		\end{align}
	\end{enumerate}
\end{lem}

\begin{proof}
	\begin{enumerate}
		\item [(a)]  First, choose {$n_*\ge 1$} large enough such that
		\begin{align}
			\int_{{n_*}}^\infty \ff^{\mathrm{in}}(x)dx\le \frac{ \mm_0(\ff^{\mathrm{in}})}{2}. \label{LBOZM 1:eq}
		\end{align}
		Next, let  $t \ge 0$ and $\tf( x)=1$ for $ x\in\mathbb{R}_+$. Then, with the help of  \eqref{nop_assump:eq}, \eqref{kernel:eq}, \eqref{tkernel:eq}, and \eqref{twf:eq}, we get
		\begin{align*}
			\frac{d}{dt} \mm_0(\ff_n(t)) = \frac{1}{2} \int_0^\infty \int_0^\infty \left[ N(x,y)-2 \right] \kk_n(x,y) \ff_n(t,x) \ff_n(t,y)\ dydx\ge0.
		\end{align*}
		Integrating with respect to time and using \eqref{LBOZM 1:eq}, we get, for $t\ge 0$
		\begin{align*}
			\mm_0(\ff_n(t))\ge 	\mm_0(\ff^{\mathrm{in}}_n)\ge \frac{ \mm_0(\ff^{\mathrm{in}})}{2}.
		\end{align*}

		\item [(b)]Let $t\in [0,T]$ and $\tf( x)= x^{-\alpha}$ for $ x\in\mathbb{R}_+$. Then, from \eqref{kernel:eq}, \eqref{neg_nop_assump:eq} and \eqref{tkernel:eq}, we get
		\begin{align}
			\frac{d}{dt}\int_0^\infty  x^{-\alpha} &\ff_n(t,x)d x \le \frac{L_{-\alpha}}{2}\int_0^\infty\int_0^\infty( x^{-\alpha}+ y^{-\alpha})\kk( x, y)\ff_n(t,x)\ff_n(t,y)d x d y\nonumber\\\
			&\le L_{-\alpha}\int_0^\infty\int_0^\infty  x^{-\alpha}\kk( x, y)\ff_n(t,x)\ff_n(t,y)d x d y \nonumber\\
			&\le \kappa L_{-\alpha}[\mm_{{{\sigma_1}}-\alpha}(\ff_n(t))\mm_{{\sigma_2}}(\ff_n(t))+\mm_{{{\sigma_2}}-\alpha}(\ff_n(t))\mm_{{\sigma_1}}(\ff_n(t))]. \label{neg ME:eq}
		\end{align}
		
	Since $0\le \sigma_1\le \sigma_2\le 1$ and $-\alpha\le \sigma_1-\alpha\le \sigma_2-\alpha\le 1$, then by  \eqref{tmassconservation:eq}, \eqref{UBZM:eq} and H\"older's inequality, we obtain the following bounds
		\begin{align*}
			\mm_{{\sigma_1}}(\ff_n(t))&\le \mm_0(\ff_n(t))^{1-\sigma_1}\mm_1(\ff_n(t))^{\sigma_1} \le \Pi_1(T)^{1-{{\sigma_1}}}\rho^{{\sigma_1}}, \\
			\mm_{{\sigma_2}}(\ff_n(t))&\le \mm_0(\ff_n(t))^{1-\sigma_2}\mm_1(\ff_n(t))^{\sigma_2}\le \Pi_1(T)^{1-{{\sigma_2}}}\rho^{{\sigma_2}}, \\
			\mm_{{{\sigma_1}}-\alpha}(\ff_n(t))&\le \mm_{-\alpha}(\ff_n(t))^{\frac{1+\alpha-{\sigma_1}}{1+\alpha}}\mm_1(\ff_n(t))^{\frac{{\sigma_1}}{1+\alpha}}\le \rho^{\frac{{\sigma_1}}{1+\alpha}} \mm_{-\alpha}(\ff_n(t))^{\frac{1+\alpha-{\sigma_1}}{1+\alpha}},\\
			\mm_{{{\sigma_2}}-\alpha}(\ff_n(t))&\le \mm_{-\alpha}(\ff_n(t))^{\frac{1+\alpha-{\sigma_2}}{1+\alpha}}\mm_1(\ff_n(t))^{\frac{{\sigma_2}}{1+\alpha}} \le \rho^{\frac{{\sigma_2}}{1+\alpha}} \mm_{-\alpha}(\ff_n(t))^{\frac{1+\alpha-{\sigma_2}}{1+\alpha}}.
		\end{align*}
		Using above estimate in \eqref{neg ME:eq}, we get
		\begin{align*}
			\frac{d\mm_{-\alpha}}{dt} (\ff_n(t))&\le c[\mm_{-\alpha}(\ff_n(t))^{\frac{1+\alpha-{\sigma_1}}{1+\alpha}}+\mm_{-\alpha}(\ff_n(t))^{\frac{1+\alpha-{\sigma_2}}{1+\alpha}}],
		\end{align*}
		where $$c:=\kappa L_{-\alpha}\max\left \{\Pi_1(T)^{1-{{\sigma_2}}}\rho^{{{\sigma_2}}+\frac{{\sigma_1}}{1+\alpha}},\Pi_1(T)^{1-{{\sigma_1}}}\rho^{{{\sigma_1}}+\frac{{\sigma_2}}{1+\alpha}}\right \}.$$
		Consequently, by Young's inequality,
		\begin{align*}
			\frac{d\mm_{-\alpha}}{dt} (\ff_n(t))&\le \frac{2c}{1+\alpha}[1+\mm_{-\alpha}(\ff_n(t))],
		\end{align*}
	from this, it follows that
		\begin{equation*}
			\mm_{-\alpha}(\ff_n(t))\le \bigg[ 1 + \mm_{-\alpha}(\ff_n^{\mathrm{in}}) \bigg] e^{\frac{2ct}{1+\alpha} } \le \bigg[ 1 + \mm_{-\alpha}(\ff^{\mathrm{in}}) \bigg] e^{\frac{2ct}{1+\alpha} }
		\end{equation*}
		for $t\ge 0$. Introducing 
		\begin{equation*}
			\Pi_{11}(T) := \bigg[ 1 + \mm_{-\alpha}(\ff^{\mathrm{in}}) \bigg] e^{\frac{2cT}{1+\alpha} },
		\end{equation*}
		which provides the estimate outlined in \Cref{lem:LB on alpha}(b).
		
		\item [(c)]	Since $-\alpha<0<{\sigma_1}$, by using \Cref{lem:LB on alpha}(a) and (c) and H\"older's inequality, we get
		\begin{align*}
			\frac{ \mm_0(\ff^{\mathrm{in}})}{2} \le \mm_{0}(\ff_n(t)) \le 	\mm_{-\alpha}(\ff_n(t))^{\frac{{\sigma_1}}{\alpha+{\sigma_1}}}  \mm_{{\sigma_1}}(\ff_n(t))^{\frac{\alpha}{\alpha+{\sigma_1}}}\le \Pi_{11}(T)^{\frac{{\sigma_1}}{\alpha+{\sigma_1}}}  \mm_{{\sigma_1}}(\ff_n(t))^{\frac{\alpha}{\alpha+{\sigma_1}}}.
		\end{align*}
		Hence, 
		\begin{align*}
			\mm_{{\sigma_1}}(\ff_n(t))\ge  \Pi_{11}(T)^{\frac{-{\sigma_1}}{\alpha}} 	\left(\frac{ \mm_0(\ff^{\mathrm{in}})}{2}\right)^{\frac{\alpha+{\sigma_1}}{\alpha}},
		\end{align*}
		which completes the proof of the \Cref{lem:LB on alpha}(c).
	\end{enumerate}
\end{proof}

\begin{proof}[Proof of \Cref{thm:existence}]
	A detailed analysis of the proofs for \Cref{lem:supermoments 1}, \Cref{lem:refine supermoments}, {\Cref{lem:ui}} and \Cref{lem:equicontinuity} indicates that these results remain applicable. The key adjustments involve replacing each reference to \Cref{lem: uniform lower bound zero moment} with \Cref{lem:LB on alpha}. As a result, the remainder of the proof for \Cref{thm:existence} follows the same trajectory as that of \Cref{thm:existence E=0}, incorporating these changes seamlessly.
\end{proof}

%
\medskip

\textbf{Acknowledgments.} 
The authors wish to express their gratitude to the Council of Scientific \& Industrial Research (CSIR), India, for granting a PhD fellowship to RGJ under Grant 09/143(0996)/2019-EMR-I. {We sincerely thank the referees for their valuable comments and suggestions.}


	\bibliography{Refs.bib}
	\bibliographystyle{abbrv}
\end{document}